\documentclass[12pt, twoside, leqno]{article}

\usepackage{amsmath,amsthm}
\usepackage{amssymb}

\usepackage{wasysym}

\usepackage{enumitem}

\usepackage{graphicx}
\usepackage{mathtools}

\usepackage[T1]{fontenc}
\usepackage{url}

\usepackage{thmtools}
\usepackage{thm-restate}

\usepackage{colortbl}
\usepackage{xcolor}
\usepackage{tikz}
\usetikzlibrary{positioning,arrows,calc}
\usetikzlibrary{decorations.pathreplacing}

\usepackage{subcaption}

\usepackage{comment}

\newtheorem{theorem}{Theorem}[section]
\newtheorem{corollary}[theorem]{Corollary}
\newtheorem{lemma}[theorem]{Lemma}
\newtheorem{proposition}[theorem]{Proposition}

\newtheorem{question}[theorem]{Question}

\theoremstyle{definition}
\newtheorem{definition}[theorem]{Definition}
\newtheorem{remark}[theorem]{Remark}
\newtheorem{example}[theorem]{Example}

\numberwithin{equation}{section}

\newcommand{\Z}{\mathbb{Z}}
\newcommand{\N}{\mathbb{N}}
\newcommand{\Q}{\mathbb{Q}}
\newcommand{\R}{\mathbb{R}}
\newcommand{\primes}{\mathbb{P}}

\newcommand{\ord}{\mathrm{ord}}
\newcommand{\tr}{\mathrm{tr}}

\newcommand{\Aut}{\mathrm{Aut}}
\newcommand{\End}{\mathrm{End}}
\newcommand{\RCA}{\mathrm{RCA}}
\newcommand{\PAut}{\mathrm{PAut}}
\newcommand{\RAut}{\mathrm{RAut}}
\newcommand{\RPAut}{\mathrm{RPAut}}
\newcommand{\Sym}{\mathrm{Sym}}
\newcommand{\Alt}{\mathrm{Alt}}

\newcommand{\id}{\mathrm{id}}

\newcommand{\GL}{\mathrm{GL}}
\newcommand{\sm}[2]{\left(\begin{smallmatrix}#1 \\ #2\end{smallmatrix}\right)}

\newcommand\xqed[1]{%
  \leavevmode\unskip\penalty9999 \hbox{}\nobreak\hfill
  \quad\hbox{#1}}
\newcommand\qee{\xqed{\fullmoon}}

\newcommand{\ctrl}[2]{#1^{#2}} 

\newcommand{\gr}[1]{{\color{green!60!black}#1}}
\newcommand{\rr}[1]{{\color{red}#1}}
\newcommand{\bl}[1]{{\color{blue!70!white}#1}}

\newcommand{\lstair}[6]{
\draw (#1-1.5, #2-0.5) -- ++(0, 1) -- ++(1, 0) -- ++(0, 1) -- ++(2, 0) -- ++(0, -1) -- ++(-1, 0) -- ++(0, -1) -- ++(-2, 0);
\node () at (#1-1,#2) {\footnotesize #3};
\node () at (#1,#2) {\footnotesize #4};
\node () at (#1,#2+1) {\footnotesize #5};
\node () at (#1+1,#2+1) {\footnotesize #6};
}
\newcommand{\rstair}[6]{
\draw (#1-1.5, #2+0.5) -- ++(0, 1) -- ++(2, 0) -- ++(0, -1) -- ++(1, 0) -- ++(0, -1) -- ++(-2, 0) -- ++(0, 1) -- ++(-1, 0);
\node () at (#1-1,#2+1) {\footnotesize #3};
\node () at (#1,#2+1) {\footnotesize #4};
\node () at (#1,#2) {\footnotesize #5};
\node () at (#1+1,#2) {\footnotesize #6};
}

\newcommand{\rlstair}[6]{
\draw (#1-1.5, #2-0.5) -- ++(0, 1) -- ++(1, 0) -- ++(0, 1) -- ++(2, 0) -- ++(0, -1) -- ++(-1, 0) -- ++(0, -1) -- ++(-2, 0);
\node () at (#1-1,#2) {\reflectbox{\footnotesize #3}};
\node () at (#1,#2) {\reflectbox{\footnotesize #4}};
\node () at (#1,#2+1) {\reflectbox{\footnotesize #5}};
\node () at (#1+1,#2+1) {\reflectbox{\footnotesize #6}};
}
\newcommand{\rrstair}[6]{
\draw (#1-1.5, #2+0.5) -- ++(0, 1) -- ++(2, 0) -- ++(0, -1) -- ++(1, 0) -- ++(0, -1) -- ++(-2, 0) -- ++(0, 1) -- ++(-1, 0);
\node () at (#1-1,#2+1) {\reflectbox{\footnotesize #3}};
\node () at (#1,#2+1) {\reflectbox{\footnotesize #4}};
\node () at (#1,#2) {\reflectbox{\footnotesize #5}};
\node () at (#1+1,#2) {\reflectbox{\footnotesize #6}};
}

\begin{document}


\baselineskip=17pt


\title{Universal groups of cellular automata}

\author{Ville Salo\\
University of Turku\\ 
E-mail: vosalo@utu.fi}

\date{}

\maketitle


\renewcommand{\thefootnote}{}

\footnote{2020 \emph{Mathematics Subject Classification}: Primary 37B10; Secondary 20B27.}

\footnote{\emph{Key words and phrases}: full shift, automorphism group, reversible cellular automata.}

\renewcommand{\thefootnote}{\arabic{footnote}}
\setcounter{footnote}{0}


\begin{abstract}
We prove that the group of reversible cellular automata (RCA), on any alphabet $A$, contains a subgroup generated by three involutions which contains an isomorphic copy of every finitely generated group of RCA on any alphabet $B$. This result follows from a case study of groups of RCA generated by symbol permutations and partial shifts (equivalently, partitioned cellular automata) with respect to a fixed Cartesian product decomposition of the alphabet. For prime alphabets, we show that this group is virtually cyclic, and that for composite alphabets it is non-amenable. For alphabet size four, it is a linear group. For non-prime non-four alphabets, it contains copies of all finitely generated groups of RCA. We also prove this property for the group generated by RCA of biradius one on any full shift with large enough alphabet, and also for some perfect finitely generated groups of RCA.
\end{abstract}

\section{Introduction}

Automorphism groups of subshifts have been a topic of much interest in recent years \cite{Ol13,SaTo15d,Sa14d,CyKr16a,CoQuYa16,CyKr16b,DoDuMaPe16,DoDuMaPe17,CyFrKrPe18,FrScTa19,Sa16a,Sa19a,BaKaSa16},
with most results dealing with either the case of highly constrained subshifts such as minimal and low-complexity subshifts, or the case of weakly constrained subshifts such as SFTs. This paper is about the second case.

Reversible cellular automata or \emph{RCA} (on a finite alphabet $A$) are the automorphisms, i.e. shift-commuting self-homeomorphisms, of the full shift $A^\Z$, and form a group denoted by $\Aut(A^\Z)$. We write this group also as $\RCA(A)$, and as $\RCA(|A|)$ up to isomorphism. Since this group is not finitely generated \cite{BoLiRu88}, from the perspective of geometric group theory it is of interest to try to understand its finitely generated subgroups. In this paper, we construct ``universal'' such subgroups, with a maximal set of finitely-generated subgroups.

A simple way to construct RCA is the technique of partitioned cellular automata. Fix a Cartesian product decomposition $A = B_1 \times B_2 \times \cdots \times B_k$ of the finite alphabet $A$. The \emph{partial shifts} shift one of the tracks with respect to this decomposition of the alphabet, e.g. identifying $x \in A^\Z$ as $(y^1, y^2, \cdots, y^k) \in B_1^\Z \times B_2^\Z \times \cdots \times B_k^\Z$ in an obvious way, we map $\sigma_1(y^1, y^2, \cdots, y^k) = (\sigma(y^1), y^2, ..., y^k)$ where $\sigma$ is the usual shift map, and similarly we allow shifting the other tracks independently. The \emph{symbol permutations} apply the same permutation of $A$ in each position of $x \in A^\Z$. These maps are reversible, and thus any composition of them is as well.

When a partial shift and a symbol permutation are composed (in some fixed order), we obtain a \emph{partitioned RCA}. In this paper, we denote the group generated by symbol permutations and partial shifts by $\PAut[B_1; B_2; ..., B_k]$ -- this group contains the partitioned RCA and their inverses, but also several other things, see Section~\ref{sec:PAut} for details. The group $\PAut[B_1; B_2; ..., B_k]$ is a subgroup of $\langle \RCA_1(A^\Z) \rangle$, the group of RCA generated by those with biradius one (meaning the CA and its inverse both have radius one). When $n_1,n_2,...,n_k \in \N$, we also write $\PAut[n_1;...;n_k]$ for the abstract group $\PAut[B_1; B_2; ..., B_k]$ where $|B_i| = n_i$, up to isomorphism.


A theorem of \cite{Ka96} shows that up to passing to a subaction of the shift (and using the induced basis for the algebra of clopen sets), all RCA essentially come from composing partial shifts and symbol permutations. The ``symbol permutations'' for a subaction of course permute longer words, and are called \emph{block permutations} in \cite{Ka96}. Such block permutations, applied with several offsets (i.e.\ conjugated by the original shift action) can build any reversible cellular automaton $f$ with zero average information flow (meaning $h_-(f) = 1$ in the morphism of \cite{Ka96}, equivalently $f$ is \emph{inert} in symbolic dynamics terminology). Partial shifts are used to eliminate the flow of information.

In the theorem of \cite{Ka96}, one needs subactions of increasing index (i.e.\ longer and longer blocks) to construct cellular automata of increasing radii, and indeed this is necessary: one can show that no finite set of block permutations and partial shifts generates a group containing all reversible cellular automata. Our main result is that for any robust enough composite alphabet $B \times C$, even without passing to a subaction of the shift, RCA in $\PAut[B; C]$ can \emph{simulate} any RCA on any alphabet in the following algebraic sense. 

\begin{definition}
Let $G$ be a group. A subgroup $H \leq G$ is \emph{universal} if there is an embedding $G \hookrightarrow H$. It is \emph{f.g.-universal} if for every finitely generated subgroup $K \leq G$ there exists an embedding $K \hookrightarrow H$.
\end{definition}

\begin{restatable}{theorem}{thmMain}
\label{thm:Main}
If $m \geq 2, n \geq 3$, $\PAut[m; n]$ is f.g.-universal in $\RCA(mn)$.
\end{restatable}

This is proved in Section~\ref{sec:Universal}. The embeddings of finitely generated subgroups are dynamical, in the sense that we concretely simulate cellular automata on encoded configurations. It is difficult (though possible) to give a good abstract formulation for a ``dynamical embedding'', so we recommend the interested reader simply read the first five paragraphs of the proof of Lemma~\ref{lem:EvensUnderFIsEnough}.

The set of finitely generated subgroups of $\RCA(A)$ does not depend on $A$ as long as $|A| \geq 2$ by \cite{KiRo90}, so when $\PAut[m; n]$ is f.g.-universal in $\RCA(mn)$, it also contains a copy of every finitely generated subgroup of $\RCA(k)$ for any other $k \in \N_+$. For the same reason, the theorem implies that for any non-trivial alphabet $A$, $\RCA(A)$ contains an f.g.-universal finitely generated subgroup since it contains a copy of each $\PAut[m; n]$ (a stronger statement about sofic shifts is given below).

The group $\RCA(A)$ is neither amenable nor locally linear when $|A| \geq 2$, so the following result shows that Theorem~\ref{thm:Main} is optimal.

\begin{restatable}{theorem}{thmMainLinAme}
\label{thm:MainLinAme}
Let $m, n \geq 1$.
\begin{itemize}
\item $\PAut[m] \cong \PAut[1; m] \cong \Z \times S_m$, while
\item if $m \geq 2, n \geq 2$ then $\PAut[m; n]$ is nonamenable, and
\item $\PAut[2; 2]$ is a linear group, while
\item if $m \geq 2, n \geq 3$ then $\PAut[m; n]$ is not a subdirect product of finitely many linear groups.
\end{itemize}
\end{restatable}

Both nonamenability and nonlinearity for $\PAut[m; n]$ for $m \geq 2, n \geq 3$ follow directly from f.g.-universality, but we give instead a uniform natural construction that proves the second and fourth item simultaneously by embedding groups of the form $\Z_k^\omega * \Z_\ell^\omega$, in Section~\ref{sec:LinAme}. Linearity of $\PAut[2; 2]$ is proved in Section~\ref{sec:Linear}, where we show that $\PAut(4) \cong \Z_2^2 \rtimes \GL(2, \Z_2[\pmb x, \pmb x^{-1}])$. 

\arraycolsep=1.2pt\def\arraystretch{1.2}
\begin{table}[h]
\begin{center}
\begin{tabular}{ |c|ccc| }
\hline
$m \geq 3$      & $\Z \times S_n$  & $\RCA(k)$ & $\RCA(k)$ \\
$m = 2$      & $\Z \times S_n$  & $\Z_2^2 \rtimes \GL(2, \Z_2[\pmb x, \pmb x^{-1}])$      & $\RCA(k)$ \\
$m = 1$      & 1   & $\Z \times S_n$     & $\Z \times S_n$  \\ 
\arrayrulecolor{black}\hline
             & $n = 1$ & $n = 2$    & $n \geq 3$  \\ 
\hline
\end{tabular}
\end{center}
\caption{The entry $G$ at $(m, n)$ means that finitely generated subgroups of $\PAut[m; n]$ are precisely the finitely generated subgroups of the group $G$; $k \geq 2$ is arbitrary.}
\label{tab:Summary}
\end{table}

Table~\ref{tab:Summary} gives the characterizations of f.g.\ subgroups of $\PAut[m; n]$, by giving, in each case, a well-known group whose f.g.\ subgroups are the same as those of $\PAut[m; n]$. 

We can also state the result in terms of alphabet size alone. Write $\PAut(A)$ for the group $\PAut[B_1; B_2; ..., B_k]$ seen through any bijection $\pi : A \to B_1 \times B_2 \times \cdots \times B_k$ where $|A| = |B_1||B_2| \cdots |B_k|$ is a full prime decomposition of $A$. The subgroup of $\RCA(A)$ obtained does not depend (even as a set) on the choice of the $B_i$ and that of $\pi$, see Section~\ref{sec:PAut}. Again up to isomorphism we write $\PAut(n)$ for the group $\PAut(A)$ where $|A| = n$.

\begin{restatable}{theorem}{thmMainSize}
\label{thm:MainSize}
Let $n \geq 2$.
\begin{itemize}
\item If $n \in \primes$, then $\PAut(n) \cong \Z \times S_n$.
\item If $n = 4$, then $\PAut(n) \cong \Z_2^2 \rtimes \GL(2, \Z_2[\pmb x, \pmb x^{-1}])$.
\item If $n \notin \primes \cup \{4\}$, then $\PAut(n)$ is f.g.-universal in $\RCA(n)$.
\end{itemize}
The group is virtually cyclic if and only if it is amenable if and only if $n \in \primes$. 
\end{restatable}


We also obtain a corollary about the group $\langle \RCA_1(n) \rangle$ of RCA generated by those with biradius one. This classifies the possible sets of f.g.\ subgroups for a cofinite set of alphabet sizes, and is proved in Theorem~\ref{thm:OptimalRadiusProof}.

\begin{theorem}
\label{thm:OptimalRadius}
$\langle \RCA_1(n) \rangle \leq \RCA(n)$ is f.g.-universal for large enough $n$.
\end{theorem}

We also construct f.g.-universal groups with small generating sets. This is proved in Theorem~\ref{thm:MinimalGenerating}.

\begin{theorem}
For any $n \geq 2$, $\RCA(n)$ has an f.g.-universal subgroup generated by three involutions, as well as one generated by two elements one of which is an involution.
\end{theorem}

As mentioned above, one motivation for proving such results is that the groups $\Aut(A^\Z)$, and more generally $\Aut(X)$ for mixing SFTs $X$, are not finitely generated, and thus do not fit very neatly in the framework of geometric group theory. Thus, it is of interest to look for finitely generated subgroups which are representative of the entire group. On the other hand, in cases where we do not obtain universality, such study provides new examples of ``naturally occurring'' finitely generated RCA groups.

The set of finitely generated subgroups of $\Aut(A^\Z)$ is relatively big: It is closed under direct and free products and finite extensions \cite{Sa16a}, contains the graph groups (a.k.a. right-angled Artin groups) \cite{KiRo90}, and contains a group not satisfying the Tits alternative \cite{Sa19a} (we give another proof in Proposition~\ref{prop:NoTits}). In the planned extended version of \cite{BaKaSa16} we prove that there is an f.g.\ subgroup with undecidable torsion problem. Since the constructions of the present paper are constructive, Theorem~\ref{thm:Main} combined with \cite{KaOl08} provides a new proof of this.\footnote{Though the extended version of \cite{BaKaSa16} is not submitted or available online, it precedes the results of this paper and uses different methods -- there the work of producing a ``generating set'' is done in the group of Turing machines, while here it is done in the group of RCA.}

We state one corollary obtained in the symbolic dynamics setting (other embedding theorems are surveyed in \cite{Sa19a}). A \emph{sofic shift} is a subshift defined by a regular language of forbidden patterns; in particular all full shifts $A^\Z$ are trivially sofic. This is proved in Theorem~\ref{thm:Sofic}. 

\begin{theorem}
Let $X$ be an uncountable sofic shift. Then the group $\Aut(X)$ has a perfect subgroup generated by six involutions containing every f.g.\ subgroup of $\Aut(A^\Z)$ for any alphabet $A$.
\end{theorem}

The reason we include ``perfect'' in this statement is that in symbolic dynamics one is often interested specifically in ``inert'' elements of $\Aut(X)$ (this means that the natural action on the dimension group \cite{LiMa95}, i.e.\ the dimension representation, is trivial), and elements of its commutator subgroup. A perfect subgroup is contained in the commutator subgroup, and if the dimension representation of the automorphism group is solvable, a perfect subgroup can only contain inert elements.



We also summarize some properties of the abstract group obtained, for easier reference.


\begin{restatable}{theorem}{abstractstatement}
There exists a finitely generated residually finite perfect group $G$ such that, letting $\mathcal{G}$ be the class of finitely generated subgroups of $G$:
\begin{itemize}
\item $G$ has decidable word problem and undecidable torsion problem, and does not satisfy the Tits alternative, and
\item $\mathcal{G}$ is closed under finite extensions, direct products and free products, and contains all f.g.\ graph groups (that is, right-angled Artin groups).
\end{itemize}
\end{restatable}

Any group with this list of properties is necessarily not a linear group over any field, contains every finite group, and every finitely generated abelian group and free group. 
We are not aware of many naturally occurring residually finite groups with such properties; for example the Tits alternative rules out linear groups, hyperbolic groups\footnote{It is not known whether all hyperbolic groups are residually finite \cite{Gr87}.} and fundamental groups of $3$-manifolds \cite{KoZa07,Fr15}, and having all finite groups as subgroups rules out automata groups.

In Section~\ref{sec:Questions}, we state some open questions about the groups $\RCA(A)$, and also 
 about the existence of (f.g.-)universal subgroups in other non-finitely generated groups of interest, namely other cellular automata groups, automata groups and the rational group, the group of Turing machines \cite{BaKaSa16}, topological full groups and (full) homeomorphism groups.

\subsection{Outline of the proof of f.g.-universality of $\PAut[m;n]$}

We outline the proof that $\PAut[m;n]$ contains an embedded copy of $\Aut(B^\Z)$ for any alphabet $B$, as long as $m$ and $n$ are large enough. Supposing $m \geq 2$ and $n \geq 3$, we use the first track, over alphabet $\{0,1,...,m-1\}$ as a ``control track''. Its value is never modified (except very temporarily), and everything interesting happens on the second, ``data track'', over alphabet $\{0,1,\ldots,n-1\}$, which we visualize as being ``below'' the control track.

First, we prove that if $w \in \{0,1,\ldots,m-1\}^*$ is an unbordered word (meaning it does not non-trivially overlap its shifted copies), then we can perform any permutation of length-$|w|$ words ``under'' the words $w$ and $ww$. Second, we prove that this allows us to simulate the action of any finitely generated group of reversible cellular automata under runs $www\ldots w$ appearing on the first track.

The first part is further proved in two stages. We recall the basics of how permutations of words can be built from permutations of subwords in Section~\ref{sec:Generators}, and from this (mostly known) theory it follows that it is sufficient to permute words of length two under $w$. The basic idea for this is that the commutator of two ``permutations conditioned on an event'' is the commutator of the permutations, conditioned on the \textit{intersection} of the events.

The basic idea for permuting words of length two under occurrences of $w$ is roughly same as in Section~\ref{sec:Generators}. This is straightforward for large $m$ and $n$, but the cases $(m,n) \in \{(2,3), (2,4), (3,3)\}$ are tricky, and we use an ad hoc argument. This is performed in Lemma~\ref{lem:FiveCase}.

To simulate finitely generated groups of reversible cellular automata under runs $www \ldots w$, we recall the trick of Kari \cite{Ka96} of representing reversible cellular automata by block permutations and partial shifts. This trick directly generalizes to finitely generated groups of reversible cellular automata. Under $w$, we can use this to simulate the action of the group.

There are some details that need to be dealt with: at boundaries of the $w$-runs we need to do something natural to obtain a homomorphism, and we use the standard trick of joining the words under $w$-runs into a conveyor belt. To deal with parity issues, we duplicate the action. This is all done in Lemma~\ref{lem:EvensUnderFIsEnough}.

\section{Definitions}
\label{sec:Conventions}

\subsection{Conventions and terminology}

Our conventions for the naturals are $0 \in \N$, $\N_+ = \N \setminus \{0\}$, and the set of primes is $\primes$. Intervals are discrete unless otherwise specified, i.e. $[a, b] = [a, b] \cap \Z$. Some basic knowledge of group theory \cite{Ro96}, symbolic dynamics \cite{LiMa95} and cellular automata is assumed, and we try to follow standard conventions.

An \emph{alphabet} is a finite set, and we mainly consider \emph{non-trivial} alphabets, namely ones with at least two elements. A \emph{subshift} is a shift-invariant closed subset of $A^\Z$ for an alphabet $A$, where the shift $\sigma : A^\Z \to A^\Z$ is $\sigma(x)_i = x_{i+1}$. Elements $x \in A^\Z$ are called \emph{configurations} or \emph{points}. 
If $X$ is a subshift, a \emph{basic cylinder} is a set of the form $[a]_i = \{x \in A^\Z \;|\; x_i = a\}$ where $a$ is in the alphabet of $X$. Basic cylinders form a subbase of the topology. 

The \emph{automorphisms} of a subshift $X$ are the self-homeomorphisms of $X$ that commute with the shift $\sigma$, and they form a group denoted by $\Aut(X)$. When $X = A^\Z$, we write $\Aut(X)$ also as $\RCA(A)$, and $\RCA(|A|)$ for the abstract group up to isomorphism.

\emph{Words} over an alphabet $A$ \cite{Lo02} form a monoid $A^*$ under concatenation, which is denoted $u \cdot v$ or $uv$. A word $u$ is \emph{$m$-unbordered} if $vu = uv' \implies |v| = 0 \vee |v| \geq m$, and \emph{unbordered} if it is $|u|$-unbordered. Configurations $x \in A^\Z$ are two-way infinite words. Often they have a periodic left and right tail, and a left tail with repeating word $u$ is written ${^\omega} u$ and a right tail as $u^\omega$. The position of the origin is either left implicit when specifying infinite words or marked with a decimal point. Finite words are $0$-indexed in formulas. In text we use the standard English ordinals, so the ``first symbol'' of a word $w$ is $w_0$ rather than $w_1$.

The clopen sets in $A^\Z$ are precisely the Boolean algebra generated by basic cylinders. We say a clopen set $F$ is \emph{$m$-unbordered} if $F \cap \sigma^i(F) = \emptyset$ for $i \in [1, m-1]$. Clearly $u$ is $m$-unbordered if and only if $[u]_i$ is an $m$-unbordered clopen set in the full shift.

For two words $u, v$ of the same length, write $D(u, v)$ for $\{i \in [0, |u|-1] \;|\; u_i \neq v_i\}$. The \emph{Hamming distance} of two words $u, v$ is $|D(u, v)|$. The Hamming distance is the path metric in the \emph{Hamming graph} (of length $n$ over alphabet $\Sigma$) whose vertices are $\Sigma^n$ and edges $(u, v)$ where $|D(u,v)| = 1$. If $a \in A$ and $u \in A^*$ write $|u|_a$ for the number of $a$-symbols in $u$.

The \emph{reversal} of a word is denoted by $w^T$ and defined by $w^T_i = w_{|w|-1-i}$. We also reverse other things such as subshifts, by reversing points in the sense $x^T_i = x_{-i}$, and cellular automata, by conjugating with the reversal map.

If $X$ and $Y$ are subshifts and $X \times Y$ their Cartesian product subshift (with the diagonal action), then $X$ and $Y$ are referred to as \emph{tracks}, and the $X$-track is also referred to as the \emph{top} track, and the $Y$-track the \emph{bottom} track. Write $\RAut(X \times Y)$ for the subgroup of $\Aut(X \times Y)$ containing those $f$ that never modify the $X$-track (i.e. $\forall x, y: \exists y': f(x, y) = (x, y')$).

An RCA $f : A^\Z \to A^\Z$ is \emph{of radius $r$} if $f(x)_0$ depends only on the word $x_{[-r,r]}$. A \emph{biradius} of a reversible cellular automaton $f$ is any number at least as large as the radii of $f$ and $f^{-1}$. The \emph{neighborhoods} are sets $N$ such that $f(x)_0$ depends only on $x|_N$, and \emph{bineighborhoods} are defined in the obvious way. 

If $N \subset \Z$ is a finite neighborhood and $A$ an alphabet, we let $\RCA_N(A)$ be the set of RCA with bineighborhood (the union of neighborhoods of the RCA and its inverse) contained in $N$. The case where $N$ is a contiguous interval is of particular interest. In the case $N = \{-r, \ldots, r\}$, that is, biradius $r$, we denote $\RCA_N(A)$ by $\RCA_r(A)$.



For two groups $G, H$, we write $H \leq G$ for the literal inclusion, and $H \hookrightarrow G$ when $H$ can be embedded into $G$.

The symmetric (resp. alternating) group on a set $A$ is $\Sym(A)$ (resp. $\Alt(A)$) and $S_n$ is the group $\Sym(A)$ for any $|A| = n$, up to isomorphism; similarly $A_n = \Alt(A)$ for $|A| = n$.

Composition of functions is from right to left and all groups (including permutation groups) act from the left unless otherwise specified. When permutations are written in cycle notation, we use whitespace or $;$ as the separator of the permutees. Usually we permute initial segments of $\N$ and elements of $\Sigma^n$ for a fixed finite alphabet $\Sigma$ and $n \in \N$.

The commutator conventions are
\[ [g, h] = g^{-1}h^{-1}gh, \;\;\; [g_1, g_2, ..., g_k] = [[g_1, g_2], g_3, ..., g_k]. \]
For $g,h$ elements of the same group, write $g^h = h^{-1} g h$. If $\phi : X \to Y$ is a bijection, we also use conjugation in the groupoid sense: if $h : Y \to Y$ is a bijection, write $h^\phi = \phi^{-1} \circ h \circ \phi : X \to X$. If $A, B$ are groups, then an $A$-by-$B$ group $G$ is one that admits an epimorphism to $B$ with kernel $A$. A virtually $H$ group (here also called a finite extension of $H$) is one that admits $H$ as a subgroup of finite index. If $A$, $B$ or $H$ are properties instead, the interpretation is existential quantification over groups with said property.

A \emph{subdirect product} of groups $G_1, \ldots, G_k$ is a subgroup of $G_1 \times \cdots \times G_k$ (one need not assume that the projection to each $G_i$ is surjective, but all our statements are true with this definition as well). A \emph{subquotient} of a group $G$ is a quotient of a subgroup.

The \emph{(transfinite) derived series} of a group $G$ is $G^{(0)} = G$, $G^{(\alpha+1)} = [G^{(\alpha)}, G^{(\alpha)}]$ for successor ordinals and $G^{(\alpha)} = \bigcap_{\beta < \alpha} G^{(\beta)}$ for limit ordinals. If this stabilizes at $G^{(k)} = 1$ for a finite ordinal $k$ (i.e. $G$ is solvable), then $k$ is called the the \emph{derived length} of $G$. The series always stabilizes at some ordinal $\alpha$, meaning $G^{(\alpha)} = G^{(\alpha+1)}$ for some minimal $\alpha$, and $G^{(\alpha)}$ is called the \emph{perfect core} of $G$. The \emph{(transfinite) lower central series} is $G_0 = G$, $G_{\alpha+1} = [G, G_{\alpha}]$ for successor ordinals and $G_\alpha = \bigcap_{\beta < \alpha} G_\beta$. This series also stabilizes at some ordinal $\alpha$, and we call $G_{\alpha}$ the \emph{hypocenter}.

A \emph{linear group} is a (not necessarily finitely generated) subgroup of a group of finite-dimensional matrices over a field, i.e. a subgroup of $\GL(n, F)$ for some field $F$ and some $n \in \N$. 

We make a few simple observations about decidability, and an informal understanding suffices: Let $\mathcal{P}$ be a family of propositions. We say $\mathcal{P}$ is \emph{semidecidable} if there exists an algorithm that, given a proposition $P$, eventually writes the answer ``yes'' if $P \in \mathcal{P}$, and eventually writes ``no'' or never writes anything if $P \notin \mathcal{P}$. We say $\mathcal{P}$ is \emph{decidable} if $\mathcal{P}$ and $\{\neg P \;|\; P \in \mathcal{P}\}$ are both semidecidable.

\subsection{$\PAut(A)$, $\PAut[B; C]$}
\label{sec:PAut}

If $B_1,B_2,...,B_k$ are finite alphabets, then $\PAut[B_1; B_2; \cdots; B_k]$ refers to the smallest subgroup of $\Aut((B_1 \times B_2 \times \cdots \times B_k)^\Z)$ containing the following maps:
The \emph{partial shifts} $\sigma_i$, $i \in [1, k]$ defined by
\[ \sigma_i(y^1, y^2, \cdots, y^k) = (y^1, y^2, ..., y^{i-1}, \sigma(y^i), y^{i+1}, ..., y^k), \]
where $\sigma : B_i^\Z \to B_i^\Z$ is the usual shift map,
and the \emph{symbol permutations} $\bar \pi$ defined by applying a permutation $\pi$ in every cell, or
\[ \bar \pi((y^1, y^2, \cdots, y^k))_j = \pi((y^1_j, y^2_j, \cdots, y^k_j)), \]
in symbols, where $\pi \in \Sym(B_1 \times B_2 \times \cdots \times B_k)$ is arbitrary. We usually identify $\bar \pi$ with $\pi$.

These maps are reversible, so $\PAut[B_1; B_2; \cdots; B_k] \leq \Aut((B_1 \times B_2 \times \cdots \times B_k)^\Z)$.

We write $\PAut(A)$ for the following subgroup of $\Aut(A^\Z)$: Let $|A| = n$ and let $n = p_1 \cdot p_2 \cdot ... \cdot p_k$ where $p_i$ are the prime factors of $n$ in any order. Pick a bijection $\psi : A \to B_1 \times B_2 \times ... \times B_k$ where $|B_i| = p_i$ for all $i$. Define $\PAut(A)$ as the group obtained by conjugating $\PAut[B_1; B_2; ...; B_k]$ through $\psi$. A priori, the resulting subgroup of $\Aut(A^\Z)$ could depend on the choice of $\psi$ and the $B_i$, but this is not the case.

\begin{lemma}
\label{lem:PAutWD}
The group $\PAut(A)$ is well-defined.
\end{lemma}

\begin{proof}
Let $\psi : A \to B_1 \times B_2 \times ... \times B_k$ and $\psi' : A \to B'_1 \times B'_2 \times ... \times B'_k$ be two bijections. By the fundamental theorem of arithmetic, and by reordering of the product (which clearly does not change the obtained subgroup of $\Aut(A^\Z)$), we may assume $|B_i| = |B'_i|$ for all $i$. Clearly the subgroup of $\Aut(A^\Z)$ obtained by using a particular bijection does not depend on the contents of the sets, but only their cardinalities, so we may hide the bijection coming from $|B_i| = |B'_i|$ and simply assume $B_i = B_i'$ for all $i$. Let $G$ and $G'$ be the two subgroups of $\Aut(A^\Z)$ generated by symbol permutations and partial shifts using the two bijections. Now, by definition, $G$ and $G'$ are conjugate subgroups of $\Aut(A^\Z)$, by the symbol permutation $\psi^{-1} \circ \psi'$ by a direct computation. This symbol permutation is in both of the groups $G$ and $G'$, so in fact the groups are equal.
\end{proof}

\section{Generators for some groups}
\label{sec:Generators}


\subsection{Controlled actions}

We begin by outlining an intuitive idea. Suppose we are dealing with a group action that is conditioned on some type of events, and write $\ctrl{g}{E}$ for the ``action of $g$ in case $E$ holds'' (this is a bijection as long as the conditioning events are not affected by the action). Then
\[ [\ctrl{g}{E}, \ctrl{h}{F}] = \ctrl{[g, h]}{E \cap F}, \]
since in the case of less than two events, the commutator cancels. When the acting group is perfect (e.g. an alternating group on at least $5$ objects), commutators $[g, h]$ are a generating set for the group, so if we can condition actions of $G$ on some set of events $\mathcal{E}$ (and their complements), we can condition them on any event in the ring of sets generated by $\mathcal{E}$, i.e. unions, intersections and relative complements of events. Typically we have a Boolean algebra of events, namely the algebra of clopen sets in some space.

The same idea can be used with $S_3$ and $S_4$, using the fact that they are not nilpotent, and their hypocenters are $A_3$ and $A_4$, respectively. Concretely, using for example the formula $[(0 \; 1 \; 2), (0 \; 1)] = (0 \; 1 \; 2)$, we can condition an even permutation on the intersection of two events, assuming one is a ``primitive event'' (so we can apply an arbitrary permutation conditioned on it), and the other is any ``composite event'' (so by induction we can apply an even permutation conditioned on it). See for example 
Lemma~\ref{lem:LowerCentralSeries} for a formal result to this effect.

 We do not give a general formalization of this idea, as often the events are entangled with whatever is being acted on, so one should rather consider this a proof technique. Informally, we refer to actions that ``depend on events'' as \emph{controlled or conditioned actions}, and use terms such as ``increase the control'' to refer to the tricks described above. The main application is to subshifts, whose Boolean algebra of clopen sets is generated by basic cylinders $[a]_i$.


\subsection{Alternating groups and $3$-hypergraphs}

The following lemma is from \cite{BoKaSa16}. A \emph{hypergraph} consists of a set of \emph{vertices} $V$ and \emph{hyperedges} $E \subset \mathcal{P}(V) \setminus \{\emptyset\}$. A hypergraph $\mathcal{G}$ is \emph{weakly connected} if the graph $\mathcal{G}'$, whose edges are those $2$-subsets of $V(\mathcal{G})$ that are contained in some hyperedge of $\mathcal{G}$, is connected.

\begin{lemma}
\label{lem:Hypergraph}
Let $\mathcal{G}$ be a hypergraph with all hyperedges of size $3$, and let $G$ be the group generated by three-cycles corresponding to the hyperedges of $\mathcal{G}$. If $\mathcal{G}$ is weakly connected, then $G = \Alt(V(\mathcal{G}))$.
\end{lemma}


\subsection{Universal families of reversible logical gates}

If you can permute two adjacent cells of words (evenly), you can permute words of any length (evenly), by the following Lemma~\ref{lem:UniversalGates} which strengthens a result of \cite{BoKaSa16}. Many results like this are known, see e.g. \cite{AaGrSc15,Bo19,Se16}, but usually (conjugation by) free reordering of wires, i.e.\ swapping the order of adjacent symbols, is allowed, so these results are not directly compatible with ours. In our application, wire reordering is not possible. (The swap of two wires is directly among the generators only if $|A| \equiv 0, 1 \bmod 4$.)

\begin{lemma}
\label{lem:UniversalGates}
Let $A$ be a finite alphabet with $|A| \geq 3$. If $n \geq 2$, then every even permutation of $A^n$ can be decomposed into even permutations of $A^2$ applied in adjacent cells. That is, the permutations
\[ w \mapsto w_0 w_1 \cdots w_{i-1} \cdot \pi(w_i w_{i+1}) \cdot w_{i+2} \cdots w_{n-1} \]
are a generating set of $\Alt(A^n)$ where $\pi$ ranges over $\Alt(A^2)$, and $i$ ranges over $0, 1,2,\ldots,n-2$.
For $|A| = 2$ the same is true when $n \geq 3$ and permutations are applied in length-$3$ subwords.
\end{lemma}

\begin{proof}
Suppose first $|A| \geq 3$, $n \geq 3$. It is enough to show that the $3$-cycles $(u; \; v; \; w)$ where for some $j$, $|\{u_j,v_j,w_j\}| = 3$, and $u_i = v_i = w_i$ for $i \neq j$, are generated. Namely, the result then follows by applying Lemma~\ref{lem:Hypergraph} to the hypergraph with vertices $A^n$ and edges $(u,v,w)$ that only differ in one position.

It is enough to show that the permutation that applies the cycle $(0 \; 1 \; 2)$ in coordinate $j$ if all other coordinates contain $0$, and is the identity otherwise, is generated. Namely, the other generators are conjugate to it or its inverse by even symbol permutations. Let us fix the coordinate $j$, and for a set of coordinates $N \not\ni j$ and permutation $\pi$, write $\ctrl{\pi}{N}$ for the permutation that applies $\pi$ in coordinate $j$ if all coordinates in $N$ contain $0$, and is the identity otherwise. We need to construct $\ctrl{(0 \; 1 \; 2)}{[0,j-1] \cup [j+1, n-1]}$.




By induction, we can assume that the map $\ctrl{(0 \; 1 \; 2)}{[j-\ell,...,j-1] \cup [j+1,...,j+r]}$, which applies $(0 \; 1 \; 2)$ at $j$ if and and only if the $\ell$ symbols to the left and $r$ symbols to the right are all $0$, is generated. By symmetry, it is enough to show that also
$\ctrl{(0 \; 1 \; 2)}{[j-\ell,...,j-1] \cup [j+1,\ldots,j+r+1]}$ is generated.


If $|A|$ is odd, define
\[ \pi = (01; \; 11)(02; \; 12)\cdots(0(|A|-1); \; 1(|A|-1)) \in \Alt(A^2), \]
and if $|A|$ is even, define 
\[ \pi = (01; \; 11)(02; \; 12)\cdots(0(|A|-1); \; 1(|A|-1))(20; \; 21) \in \Alt(A^2). \]
In each case, $\pi$ has the property that, when applied to a word $ab$, if $a = 0$ then the value of $a$ changes if and only if $b \neq 0$, and it always changes to $1$ in this case.

Let $\psi$ be the map that applies $\pi$ successively in the subwords
\[ [j+r,j+r+1], [j+r-1, j+r], \ldots, [j+1, j+2]. \]
Observe that if $w_{j-\ell,\ldots,j-1} w_{j+1, \ldots,j+r} = 0^{\ell + r}$, then $\psi(w)_{j+1} \in \{0,1\}$ and we have $\psi(w)_{j+1} = 1 \iff w_{j+r+1} \neq 0$.

Let $\beta$ apply the permutation $(00; \; 10)(02; \; 12)$ at $[j,j+1]$. Note that $\beta^\psi$ does not modify any coordinate other than $j$, i.e.\ the effect of $\psi$ is cancelled after applying $\beta$. We claim that we have
\[ [\ctrl{(0 \; 1 \; 2)}{[j-\ell,j-1] \cup [j+1,...,j+r]}, \beta^\psi] = \ctrl{(0 \; 1 \; 2)}{[j-\ell,j-1] \cup [j+1,j+r+1]}. \]
To see this, observe that if the coordinates in $[j-\ell,j-1] \cup [j+1,j+r+1]$ all contain $0$, then the commutator $[(0 \; 1 \; 2), (0 \; 1)] = (0 \; 1 \; 2)$ is applied at $j$, and no other coordinate is modified. If some coordinate in $[j-\ell,j-1] \cup [j+1,...,j+r]$ is nonzero, this fact is not changed by $\beta^\psi$, so $\ctrl{(0 \; 1 \; 2)}{[j-\ell,j-1] \cup [j+1,...,j+r]}$ has no effect, and the effect of $\beta^\psi$ cancels. If all coordinates of $[j-\ell,j-1] \cup [j+1,...,j+r]$ contain $0$ but the value at $j+r+1$ is not $0$, then $\beta^\psi$ has no effect, since just before $\beta$ is applied, $\psi$ has propagated a $1$-symbol to the coordinate $j+r+1$.

Suppose then $|A| = 2, n \geq 3$. Then it is essentially classical that the set of all even permutations of $A^4$ generates all even permutations of $A^n$ for any $n$ (swaps, flips and the Toffoli gate $(a,b,c) \mapsto (a,b,c + ab)$ are even as permutations of $A^4$), and a quick search in GAP \cite{GAP} shows that the set of all even permutations of $A^4$ is generated by the even permutations of $A^3$.
\end{proof}

The lemma does not hold for $|A| = 2$ and permutations applied in adjacent cells: all permutations of $A^2$ are affine for the natural linear structure of $A^2 \cong \Z_2^2$, so they will also give only affine maps with respect to the natural linear structure of $A^n$. In fact, they do not generate all even permutations of $A^3$.

As hinted by the title of the section, a typical and useful way to think of permutations applied to subwords is as ``reversible logical gates''. One can draw reversible gates in picture form by having a ``wire'' for each $i \in \{0, \ldots, n-1\}$, and the $i$th wire carries a signal corresponding to the symbol $w_i \in A$. A permutation $\pi$ applied to consecutive wires $\{i, i+1, \ldots, i+j-1\}$ is visualized as a box labeled with the corresponding permutation of $\Sym(A^j)$, and is thought of as a logical gate acting on the signals. The special gate which performs the operation $ab \leftrightarrow ba$ on $a,b \in A$ can be represented as a reordering of wires (the braiding of the wires carries no meaning).

We say a family of gates is \emph{universal} if it generates all the even gates on $A^n$ for large enough $n$. Combining the previous lemma with any standard set of generators for $\Alt(A^2)$, we obtain a set of two gates that generates all other gates. It is well-known that as $n$ tends to infinity, the fraction of pairs $(g, h) \in \Alt(k)$ with $\langle g, h \rangle = \Alt(k)$ tends to $1$ \cite{Di69}, so almost any two even random permutations of $A^2$ form a universal family of reversible gates. We conjecture that a single gate suffices for $n$ large enough.

\section{Structure and universality of $\PAut[...]$-groups}

We prove Theorem~\ref{thm:Main} in Section~\ref{sec:Universal}. Theorem~\ref{thm:MainLinAme} is a combination of Lemma~\ref{lem:Prime}, Theorem~\ref{thm:Four} and Theorem~\ref{thm:NonLinearNonAmenable}, which are proved in sections~\ref{sec:Prime},~\ref{sec:Linear} and~\ref{sec:LinAme}, respectively. In addition to the results mentioned, we discuss some basic structural properties of subgroups which arise in the course of the proof.

\subsection{Universal groups}
\label{sec:Universal}

In this section, we perform the main engineering task of building copies of every finitely generated group of RCA in the $\PAut[B; C]$ groups. 

\begin{definition}
Suppose $F \subset B^\Z$ is an $n$-unbordered clopen set and $\pi : C^n \to C^n$ is a permutation. Then define $\ctrl{\pi}{F} \in \Aut((B \times C)^\Z)$ by
\[ \ctrl{\pi}{F}(x, y)_j = \left\{\begin{array}{ll}
(x_j, \pi(y_{[j-i,j-i+n-1]})_i) & \mbox{if } i \in [0,n-1], \sigma^{j-i}(x) \in F \\
(x_j, y_j) & \mbox{if } \forall i \in [0,n-1]: \sigma^{j-i}(x) \notin F. \\
\end{array}\right. \]
\end{definition}


The map $\ctrl{\pi}{F}$ performs the permutation $\pi$ on the bottom track under every occurrence of $F$ on the top track. One should think of this as a conditional application of $\pi$ on the bottom track, where the condition is that the top track contains a point that is in $F$. The definition makes sense, since due to the fact $F$ cannot overlap a translate of itself by less than $n$ steps (by $n$-unborderedness), permutations can unambiguously modify a contiguous interval of $n$ cells to the right of the place where $F$ occurs. 

\begin{example}
\label{ex:f}
Let $f = \ctrl{(\gr{00}; \gr{10}; \gr{01})}{[\bl{01}]_0}$. To apply $f$, locate occurrences of $\bl{01}$ on the top track, and permute the words under the occurrences according to the permutation $(\gr{00}; \gr{10}; \gr{01})$:
\begin{align*}
f & \left( \begin{matrix}
...0100111001001001001000110010010...\\
...0101110011010011010101001001010...
\end{matrix} \right) = \\
f & \left( \begin{matrix}
...\bl{01}0\bl{01}110\bl{01}0\bl{01}0\bl{01}0\bl{01}00\bl{01}10\bl{01}0\bl{01}0...\\
...\gr{01}0\rr{11}100\rr{11}0\gr{10}0\rr{11}0\gr{10}10\gr{10}01\gr{00}1\gr{01}0...
\end{matrix} \right) = \\
& \hspace{9pt} \begin{matrix}
...0100111001001001001000110010010...\\
...0001110011001011001100101101000...
\end{matrix}
\end{align*}
where we write occurrences of the controlling clopen set $[\bl{01}]_0$ in blue, words modified by the permutation in green, and the fixed points of the permutation (to which it is nevertheless applied) in red.

One can also extract an explicit local rule:
\[ \begin{tikzpicture}[scale = 0.5] \draw (0,0) grid (3,1); \draw (1,-1) rectangle (2,0); \node at (1.5,0.5) {$\begin{smallmatrix}0\\0\end{smallmatrix}$}; \node at (2.5,0.5) {$\begin{smallmatrix}1\\0\end{smallmatrix}$}; \node at (1.5,-0.5) {$\begin{smallmatrix}0\\1\end{smallmatrix}$}; \end{tikzpicture}\;\;\;\;\begin{tikzpicture}[scale = 0.5] \draw (0,0) grid (3,1); \draw (1,-1) rectangle (2,0); \node at (1.5,0.5) {$\begin{smallmatrix}0\\0\end{smallmatrix}$}; \node at (2.5,0.5) {$\begin{smallmatrix}1\\1\end{smallmatrix}$}; \node at (1.5,-0.5) {$\begin{smallmatrix}0\\0\end{smallmatrix}$}; \end{tikzpicture}\;\;\;\;\begin{tikzpicture}[scale = 0.5] \draw (0,0) grid (3,1); \draw (1,-1) rectangle (2,0); \node at (1.5,0.5) {$\begin{smallmatrix}0\\1\end{smallmatrix}$}; \node at (2.5,0.5) {$\begin{smallmatrix}1\\0\end{smallmatrix}$}; \node at (1.5,-0.5) {$\begin{smallmatrix}0\\0\end{smallmatrix}$}; \end{tikzpicture}\;\;\;\;\begin{tikzpicture}[scale = 0.5] \draw (0,0) grid (3,1); \draw (1,-1) rectangle (2,0); \node at (0.5,0.5) {$\begin{smallmatrix}0\\0\end{smallmatrix}$}; \node at (1.5,0.5) {$\begin{smallmatrix}1\\0\end{smallmatrix}$}; \node at (1.5,-0.5) {$\begin{smallmatrix}1\\0\end{smallmatrix}$}; \end{tikzpicture}\;\;\;\;\begin{tikzpicture}[scale = 0.5] \draw (0,0) grid (3,1); \draw (1,-1) rectangle (2,0); \node at (0.5,0.5) {$\begin{smallmatrix}0\\0\end{smallmatrix}$}; \node at (1.5,0.5) {$\begin{smallmatrix}1\\1\end{smallmatrix}$}; \node at (1.5,-0.5) {$\begin{smallmatrix}1\\0\end{smallmatrix}$}; \end{tikzpicture}\;\;\;\;\begin{tikzpicture}[scale = 0.5] \draw (0,0) grid (3,1); \draw (1,-1) rectangle (2,0); \node at (0.5,0.5) {$\begin{smallmatrix}0\\1\end{smallmatrix}$}; \node at (1.5,0.5) {$\begin{smallmatrix}1\\0\end{smallmatrix}$}; \node at (1.5,-0.5) {$\begin{smallmatrix}1\\1\end{smallmatrix}$}; \end{tikzpicture} \] 

In all nonspecified cases we output the contents of the central cell.
\qee
\end{example}


\begin{definition}
\label{def:Controlled}
Let $X$ be a subshift and $G$ a group acting on a set $A$. For a clopen set $C \subset X$ and $g \in G$, define $\ctrl{g}{C} : X \times A \to X \times A$ by
\[ \ctrl{g}{C}(x,a) = \left\{\begin{array}{ll}
(x,ga) & \mbox{if } x \in C \\
(x,a) & \mbox{otherwise.}\end{array}\right. \]
Define the shift by $\sigma(x,a) = (\sigma(x),a)$ where $\sigma$ denotes both the new and the usual shift map. We denote the group generated by these maps by $\ctrl{G}{X}$. We denote by $P(X, G)$ the subgroup generated by the shift on $X$ and maps $\ctrl{g}{C}$ where $g \in G$ and $C$ is a basic cylinder.
\end{definition}

Note that $P(X, G)$ is finitely generated, since $\ctrl{g}{[a]_i} = (\ctrl{g}{[a]_0})^{\sigma^i}$.
The notation $P(X, G)$ is by analog with the `$P$' in $\PAut$, as these groups can be simulated rather transparently with elements of $\PAut$. See Section~\ref{sec:ModOneTrack} 
for some basic observations about these groups.

\begin{lemma}
\label{lem:LowerCentralSeries}
Let $X \subset \Sigma^\Z$ be a  subshift and $G$ a group acting on a set $A$. Then for all clopen $C$, $P(X, G)$ contains $\ctrl{g}{C}$ for all $g$ in the hypocenter of $G$. 
\end{lemma}

\begin{proof}
It is enough to prove this for cylinders, i.e. $C = [w]_m$ for a word $w$ and $m \in \Z$. This is true by assumption if $C$ is a basis set. 
Let then $C = [wa]_m$ where $a \in \Sigma$. If $h$ is in the hypocenter, then $h = [h_1, g_1][h_2, g_2]...[h_j, g_j]$ for some $h_i$ in the hypocenter and $g_i$ in $G$. It is thus enough to show that $\ctrl{[h_i, g_i]}{[wa]_m} \in P(X, G)$. It is easy to verify that
\[ [ \ctrl{h_i}{[w]_m}, \ctrl{g_i}{[a]_{m+|w|}}] =  \ctrl{[h_i, g_i]}{[w]_m \cap [a]_{m+|w|}} = \ctrl{[h_i, g_i]}{[wa]_m}. \]
\end{proof}

The following lemma separates the $\PAut[2; 2]$ case from others, by finding a large locally finite subgroup in $\PAut[B \times C]$. (The conclusion is true also for $|C| = 2$, but is trivial in that case.) 

\begin{lemma}
\label{lem:ControlGrowing}
Let $|B|, |C| \geq 2$. Then for every even permutation $\phi$ of $C$ and any clopen $F \subset B^\Z$, $\ctrl{\phi}{F}$ is in $\PAut[B; C]$. 
\end{lemma}

\begin{proof}
For every $n$, the hypocenter of $S_n$ is $A_n$. It is easy to see that the partial shift on either track, together with symbol permutations that only modify the bottom track, implement the group $P(B^\Z, G)$ in a natural way where $G = S_{|C|}$, and the claim follows from the previous lemma. 
\end{proof}


\begin{lemma}
\label{lem:FiveCase}
Let $|B| \geq 2, |C| \geq 3$. Then for any $n$-unbordered clopen set $F \subset B^\Z$, $\ctrl{\pi}{F} \in \PAut[B; C]$ for every $\pi \in \Alt(C^n)$.
\end{lemma}

\begin{proof}
We may assume $n \geq 2$, since $n = 1$ is covered by the previous lemma.
Any clopen set $F$ is a union of disjoint basic cylinders $[u]_i$, and it follows from the assumption that the word $u$ is necessarily $n$-unbordered for each $u$ appearing in this decomposition of $F$. We can take each $i$ and the lengths $|u|$ to be equal, and if
$F = \bigcup_{j = 1}^\ell [u_j]_i$ for finitely many distinct words $u_j \in B^m$, then the union is disjoint and
\[ \ctrl{\pi}{F} = \ctrl{\pi}{[u_\ell]_i} \circ \cdots \circ \ctrl{\pi}{[u_1]_i} \]
for any $\pi \in \Alt(C^n)$, because by the assumption that $F$ is $n$-unbordered, each coordinate can be affected by at most one of these $\ell$ applications of $\pi$. By conjugation with the shift, it is enough to show that $\ctrl{\pi}{[u]_0} \in \PAut[B; C]$ for any $n$-unbordered word $u$ and any $\pi \in \Alt(C^n)$.

We may suppose $B = \{0,\ldots,|B|-1\}, C = \{0,\ldots,|C|-1\}$. Let $(x, y)$ stand for some configuration in $(B \times C)^\Z$. By Lemma~\ref{lem:ControlGrowing}, $\ctrl{\pi}{[u]_i} \in \PAut[B; C]$ for all $\pi \in \Alt(C)$. Since $u$ is $n$-unbordered, it follows that the maps $\ctrl{\psi}{[u]_0}$, where $\psi = \pi_1 \times \pi_2 \times \cdots \times \pi_n$ is a Cartesian product of $n$ even symbol permutations (applied to consecutive symbols), are in $\PAut[B; C]$.

We claim that it is enough to show $\ctrl{(00; \; 10; \; 20)}{[u]_0}$ is in $\PAut[B; C]$. To see this, observe that then also $\ctrl{(00; \; 10; \; 20)}{[u]_j} \in \PAut[B; C]$ by conjugation by partial shifts. By symmetry, also $\ctrl{(00; \; 01; \; 02)}{[u]_j} \in \PAut[B; C]$. Since we can perform even symbol permutations in any coordinate under occurrences of $u$, it is easy to see that the sets $\{v_1, v_2, v_3\}$, $v_i \in C^2$, such that $\ctrl{(v_1; \; v_2; \; v_3)}{[u]_j} \in \PAut[B; C]$, form the hyperedges of a weakly connected hypergraph. Thus, we can perform any even permutation of $C^2$ in any two consecutive symbols under each occurrence of $u$ by Lemma~\ref{lem:Hypergraph}. By Lemma~\ref{lem:UniversalGates}, we can then perform any even permutation in each segment of length $n$ under every occurrence of $u$. Note that by $n$-unborderedness, these permutations indeed happen in disjoint segments of $y$, for distinct occurrences of $u$ in $x$.

Suppose first that $|B|$ is even (the argument is slightly cleaner in this case). Then we claim that the map $f$ defined by $f(x, y)_i = (x_i, y_i)$ if $y_{i+1} \neq 0$, $f(x, y)_i = (x_i, \pi(y_i))$ if $y_{i+1} = 0$, is in $\PAut[B; C]$ where $\pi = (1 \; 2)$.

We claim that
\[ f = (\sigma_1^{-1} \circ \ctrl{(1 \; 2)}{[E]_0} \circ \sigma_1 \circ (\ctrl{\psi}{[0]_0})^\updownarrow)^2, \]
where $\psi = (0 \; 1)(2 \; 3) \cdots ((|B|-2) \; (|B|-1))$, $E = \{0,2,4,...,|B|-2\} \subset B$, and ${\updownarrow} : (B \times C)^\Z \to (C \times B)^\Z$ exchanges the tracks. Conjugation by $\updownarrow$ is performed in the groupoid sense, and means that we modify the top track conditioned on the bottom track. To see that the formula holds, observe that since the set of positions where $0$ occurs in $y$ never changes, the effect on $x$ is cancelled. If $y_{i+1} = 0$, then the symbol at $x_i$ will be even during exactly one of the two applications, while otherwise it is even either zero times or two times, and the flip cancels out.

Then consider $[f, \ctrl{(0 \; 1 \; 2)}{[u]_0}]^2$. Since $[(1,2), (0,1,2)]^2 = (0,1,2)$, it applies the permutation $(0 \; 1 \; 2)$ at $y_i$ at least if $x_{[i,i+|u|-1]} = u$ and $y_{i + 1} = 0$, which is what we want. Let us analyze its side-effects. If $x_{[i,i+|u|-1]} = u$ and $y_{i + 1} \neq 0$, then since $n \geq 2$, $y_{i + 1}$ is nonzero after all partial applications (since $u$ is $n$-unbordered and $f$ does not modify the set of coordinates where $0$ occurs on the bottom track), so in this case the rotation $(0 \; 1 \; 2)$ cancels, and $y_i$ retains its value. This means that if $x_{[i,i+|u|-1]} = u$, the modification of the coordinate $y_i$ is correct.

Suppose next that $x_{[i,i+|u|-1]} \neq u$ and $x_{[i+1,i+|u|]} \neq u$. In this case, $\ctrl{(0 \; 1 \; 2)}{[u]_0}$ does not modify the value of $y_i$ or $y_{i+1}$, and a short calculation shows its conjugate by $f$ does neither, so $[f, \ctrl{(0 \; 1 \; 2)}{[u]_0}] = (\ctrl{(0 \; 1 \; 2)}{[u]_0})^f \circ \ctrl{(0 \; 1 \; 2)}{[u]_0}$ (and thus its square) does not change the value of $y_i$.

Suppose then that $x_{[i+1,i+|u|]} = u$ (so $x_{[i,i+|u|-1]} \neq u$ since $u$ is $n$-unbordered). Suppose first that $y_{i+2} \neq 0$. Then the application of $[f, \ctrl{(0 \; 1 \; 2)}{[u]_0}]$ does not modify $y_{i+1}$ by the previous arguments, and its only possible effect is an application of $(1 \; 2)$ at $y_i$, so this effect cancels when we take the square.

Consider then the case $x_{[i+1,i+|u|]} = u$ and $y_{i+2} = 0$. In this case, a short calculation shows that an application $[f, \ctrl{(0 \; 1 \; 2)}{[u]_0}]^2$ flips $y_i$ if $y_{i+1} \in \{0,1\}$. Since it also rotates $y_{i+1}$ by $(0 \; 2 \; 1)$, the square applies the flip $(1 \; 2)$ at $y_i$ if and only if $y_{i+1} \in \{0, 1\}$. We conclude that this is the only undesired side-effect of $[f, \ctrl{(0 \; 1 \; 2)}{[u]_0}]^2$.

We now deal with the side-effects, i.e.\ coordinates $y_i$ where $x_{[i+1,i+|u|]} = u$, $y_{i+1} \in \{0,1\}$ and $y_{i+2} = 0$. Let us continue by applying
\[ ([f, \ctrl{(0 \; 1 \; 2)}{[u]_0}]^2)^{\ctrl{(0 \; 2 \; 1)}{[u]_{0}}}, \]
i.e. we apply the same map, but conjugated by the application of $(0 \; 2 \; 1)$ at coordinates $i+1$ such that $x_{[i+1,i+|u|]} = u$. The effect on $y_{i+1}$ is as above, namely rotation by $(0 \; 1 \; 2)$, since rotations form an abelian group. Thus, in total we perform $(0 \; 2 \; 1)$ at $y_{i+1}$. But now at $y_i$ we actually perform the flip $(1 \; 2)$ under the exact same condition on the original value of $y_1$, i.e. $y_1 \in \{1, 2\}$, since before the second application, we rotated it back to its original value. Thus this undesired flip is undone.

Repeating all of the above twice, we perform $(0 \; 1 \; 2)$ at $y_i$ under the same condition $y_{i+1} = 0$, $x_{[i,i+|u|-1]} = u$. In other words,
\[ (([f, \ctrl{(0 \; 1 \; 2)}{[u]_0}]^2)^{\ctrl{(0 \; 2 \; 1)}{[u]_{0}}} \circ [f, \ctrl{(0 \; 1 \; 2)}{[u]_0}]^2)^2 = \ctrl{(00; \; 10; \; 20)}{[u]_0} \]
is in $\PAut[B; C]$, and the result follows from Lemma~\ref{lem:UniversalGates} as explained above.

Next, suppose $|B|$ is arbitrary, let $a \neq u_1$ and consider the definition
\[ f' = (\sigma_1^{-1} \circ \ctrl{(1 \; 2)}{[a]_0} \circ \sigma_1 \circ (\ctrl{\psi}{[0]_0})^\updownarrow)^{|B|}, \]
and $\psi = (0 \; 1 \; 2 \; \cdots \; (|B|-1))$. This map applies $(1 \; 2)$ at $y_i$ iff $y_{i+1} = 0$ or $x_{i+1} = a$. We can repeat the previous argument almost verbatim.

Consider $[f', \ctrl{(0 \; 1 \; 2)}{[u]_0}]^2$. It performs $(0 \; 1 \; 2)$ at $y_i$ if $x_{[i,i+|u|-1]} = u$ and $y_{i + 1} = 0$.  If $x_{[i,i+|u|-1]} = u$ and $y_{i+1} \neq 0$, $y_i$ retains its value.
Again the only coordinates $y_i$ where there might be side-effects are ones where $x_{[i+1,i+|u|]} = u$. In such a coordinate, we apply the flip $(1 \; 2)$ if and only if either $y_{i+2} = 0$ and $y_i \in \{0, 1\}$, or if $u_0 = a$.

Whether or not $u_0 = a$, as in the case when $|B|$ is even,
\[ (([f', \ctrl{(0 \; 1 \; 2)}{[u]_0}]^2)^{\ctrl{(0 \; 2 \; 1)}{[u]_{0}}} \circ [f', \ctrl{(0 \; 1 \; 2)}{[u]_0}]^2)^2 \]
is precisely the desired map $\ctrl{(00; \; 10; \; 20)}{[u]_0}$
\end{proof}

\begin{remark}
The two cases depending on the parity of $|B|$ are really about the two cases $(|B|, |C|) \in \{(2, 3), (3, 3)\}$, which were solved last. For larger $|C|$, we can separate data and control, and for example for $|C| \geq 6$ (and any $|B| \geq 2$), since $\Alt(C \setminus \{0\})$ is perfect, one can rather directly write a formula for an arbitrary even permutation of $C \setminus \{0\}$ at $y_i$ controlled by $x_{[i,i+|u|-1]} = u$ and $y_{i+1} = 0$, without side effects. After this, one again concludes by Lemma~\ref{lem:ControlGrowing} and Lemma~\ref{lem:UniversalGates}.
\end{remark}

\begin{lemma}
\label{lem:EvensUnderFIsEnough}
Let $|B|, |C| \geq 2$ and $A = B \times C$, and let $G \leq \Aut(A^\Z)$. If $r \geq 1$ and there is an unbordered word $w$ of length $\ell \geq 24r$ such that the maps $\ctrl{\pi}{[w]_i}$ and $\ctrl{\pi}{[ww]_i}$ are in $G$ for all $\pi \in \Alt(C^\ell)$ and $i \in \Z$, then $\langle \RCA_r(C) \rangle \hookrightarrow G$.
\end{lemma}

\begin{proof}
%
Let us assume $\ell = 24r$; for $\ell > 24r$ we can simply ignore the $C$-symbols under the length $\ell - 24r$ suffix of $w$, which is only a notational complication (this cannot make even permutations odd). We first associate to any $f \in \Aut(C^\Z)$ (with any radius) an element $\hat f \in \RAut((B \times C)^\Z)$ which simulates the action of $f$ in a natural way, so that $f \mapsto \hat f$ is an embedding. See Figure~\ref{fig:Lemma7Figure} for an illustration of the procedure.

The map $\hat f$ is defined as follows: Suppose $(x, y) \in B^\Z \times C^\Z$ and consider an occurrence of $w^m$ in $x$ which is not part of an occurrence of $w^{m+1}$. Note that points $x$ with the property that every maximal run of $w$s is finite are dense, so it is enough to define $\hat f$ uniformly continuously on such $(x,y)$ and extend by continuity. We split the subword of $y$ under the occurrence of $w^m$ into $u_1 v_1 u'_1 v'_1 \cdot u_2 v_2 u'_2 v'_2 \cdots u_m v_m u'_m v'_m$ where $|u_i| = |v_i| = |u'_i| = |v'_i| = 6r$ for all $i$.

The application of $\hat f$ will be defined for $f$ of any radius, but let us already address what will happen when the biradius is at most $r$. When $f$ has biradius at most $r$, we will be able to construct $\hat f$ (which is defined below) inside $G$ by performing a sequence of operations that changes the words $u_i$ and $v_i$, by applying permutations to the subwords $u_iv_i$ and the (non-contiguous) subwords $v_{i-1}u_i$ below the occurrence of $w^m$. The words $u'_i$ and $v'_i$ are changed exactly the same way, i.e.\ when we apply a permutation to the word $u_i v_i$, we apply the same permutation to $u'_i v'_i$, and a permutation applied to $v_{i-1}u_i$ is also applied to $v'_{i-1}u'_i$. The main simulation happens on the words $u_i$ and $v_i$, while the purpose of the primed versions is simply to ensure that all the permutations performed are even: for any permutation $\pi : X \to X$, the diagonal permutation $\pi \times \pi : X \times X \to X \times X$ is even.

We think of $u_i$ as being on top of the word $v_i$, and think of the boundaries of the maximal run $w^m$ as completing the top and bottom word into a conveyor belt; similarly for the primed words $u_i', v_i'$. Accordingly, to define $\hat f$, we apply $f$ to the periodic point $(u_1u_2 \cdots u_m(v_m)^T(v_{m-1})^T\cdots(v_1)^T)^\Z$ and decode the contents of $[0, 12rm]$ into the new contents below the occurrence of $w^m$; similarly for the primed words. Denote the new configuration below $w^m$ as $\bar u_1 \bar v_1 \bar u'_1 \bar v'_1 \cdot \bar u_2 \bar v_2 \bar u'_2 \bar v'_2 \cdots \bar u_m \bar v_m \bar u'_m \bar v'_m$.

This defines the global rule of $\hat f$ uniquely, as the unique continuous extension, and it is easy to see that $\hat f$ is always an automorphism (since $\hat f^{-1}$ is an inverse). If the biradius of $f$ is $r'$, then that of $\hat f$ is $4r' + \ell$ where the factor $4$ comes from skipping over words representing contents of other simulated tapes, e.g.\ skipping over $v_i, u_i', v_i'$ when rewriting $u_i$, and $\ell$ is needed because we need to know whether the sequence of $w$s continues. Since the word to which $f$ is applied only depends on $x$, and we are directly simulating the action of $f$ on an encoded configuration, the map $f \mapsto \hat f$ is a homomorphism, and since $w^m$ can appear in $x$ for arbitrarily large $m$, this is an embedding of $\Aut(C^\Z)$ into $\Aut(A^\Z)$. See \cite{Sa16a} for more detailed explanations of similar arguments.

Now, we show that for any $f \in \RCA_r(C)$, the cellular automaton $\hat f$ is indeed in $G$. 


We now recall the concept of stairs from \cite{Ka96}. Define $L \subset C^{4r}$ as the left stairs of $f$, i.e. the possible contents
\tikz[scale=0.3,baseline=-3]{
\draw (0,0) -- (0,1) -- (2,1) -- (2,0) -- (3,0) -- (3,-1) -- (1,-1) -- (1,0) -- cycle;
\node (u) at (1,0.5) {u}; \node (v) at (2,-0.5) {v};}
of stairs in spacetime diagrams (where the arrow of time points down), or in symbols
\[ L = \{uv \in C^{4r} \;|\; u, v \in C^{2r}, \exists x \in C^\Z: x_{[0,2r-1]} = u, f(x)_{[r,3r-1]} = v \}, \]
and $R \subset C^{4r}$ the right stairs of $f$ defined symmetrically.

Then $|L||R| = |C|^{6r}$ by the argument of \cite{Ka96}, namely the local rules of $f$ and $f^{-1}$ set up an explicit bijection between suitably concatenated left and right stairs and words of length~$6r$. Define $\gamma_L : C^{6r} \to L$ and $\gamma_R : C^{6r} \to R$ for the maps which extract the left and right stair corresponding to a word, and $\dot\gamma_L : C^{6r} \to \dot L$ and $\dot\gamma_R : C^{6r} \to \dot R$ for the corresponding versions for $f^T$, writing $\dot L$ and $\dot R$ for the left and right stairs of $f^T$ (recall that we define $f^T(x) = f(x^T)^T$ where for configurations $x$ we have $x^T_i = x_{-i}$).

The left stairs of $f^T$ are in bijection with the right stairs of $f$ and vice versa: we have $\dot\gamma_L = \gamma_R(w^T)^T$ in a natural sense. Therefore we have $|L||\dot L| = |L||R| = |C|^{6r}$ and similarly for right stairs. Let $\alpha_L : L \times \dot L \to C^{6r}$ and $\alpha_R : R \times \dot R \to C^{6r}$ be any bijections. 

Define also the maps $\beta_L, \beta_R : C^{6r} \to C^{3r}$ which simply extract the left and right halves of a word. 

We now do a sequence of rewrites. First, for all $i$ (simultaneously) we do
\begin{align*}
&u_iv_iu'_iv'_i \mapsto \\
& \alpha_L(\gamma_L(u_i), \dot\gamma_L(v_i)) \alpha_R(\gamma_R(u_i), \dot\gamma_R(v_i)) \cdot \alpha_L(\gamma_L(u'_i), \dot\gamma_L(v'_i)) \alpha_R(\gamma_R(u'_i), \dot\gamma_R(v'_i)) \mapsto \\
& \alpha_L(\gamma_L(u_i), \dot\gamma_L(v_i)) \alpha_L(\gamma_L(u'_i), \dot\gamma_L(v'_i)) \cdot \alpha_R(\gamma_R(u_i), \dot\gamma_R(v_i)) \alpha_R(\gamma_R(u'_i), \dot\gamma_R(v'_i)),
\end{align*}
which can be performed by applying a suitable even permutation on the bottom track, conditioned on having $w$ on the top track. To see that this permutation is even, observe that the first permutation is diagonal (i.e. of the form $\pi \times \pi$ for a permutation $\pi$) and the second is even as the words $u^i,u'^i,v^i,v'^i$ are of even length (so in fact any permutation of the order of the words is even).

Now ``between'' consecutive occurrences of $w$ for $1 \leq i < m$, i.e.\ in the middle of each occurrence of $ww$, do
\begin{align*}
&\alpha_R(\gamma_R(u_i), \dot\gamma_R(v_i)) \alpha_R(\gamma_R(u'_i), \dot\gamma_R(v'_i)) \cdot \\
& \hspace{1cm} \alpha_L(\gamma_L(u_{i+1}), \dot\gamma_L(v_{i+1})) \alpha_L(\gamma_L(u'_{i+1}), \dot\gamma_L(v'_{i+1})) \mapsto \\
& \alpha_R(\gamma_R(u_i), \dot\gamma_R(v_i)) \alpha_L(\gamma_L(u_{i+1}), \dot\gamma_L(v_{i+1})) \cdot \\
&\hspace{1cm} \alpha_R(\gamma_R(u'_i), \dot\gamma_R(v'_i)) \alpha_L(\gamma_L(u'_{i+1}), \dot\gamma_L(v'_{i+1})) \mapsto\\
& \beta_R(\bar u_i) \beta_L(\bar u_{i+1}) \beta_R(\bar v_i) \beta_L(\bar v_{i+1}) \cdot \beta_R(\bar u_i') \beta_L(\bar u_{i+1}') \beta_R(\bar v_i') \beta_L(\bar v_{i+1}') \mapsto \\
& \beta_R(\bar u_i) \beta_R(\bar v_i) \beta_R(\bar u_i') \beta_R(\bar v_i') \cdot \beta_L(\bar u_{i+1}) \beta_L(\bar v_{i+1})  \beta_L(\bar u_{i+1}') \beta_L(\bar v_{i+1}').
\end{align*}

In the second transformation in the above formula, we use the fact that the word $\beta_R(\bar u_i) \beta_L(\bar u_{i+1})$ can be determined from $\gamma_R(u_i)$ and $\gamma_L(u_{i+1})$ (and indeed the correspondence is a bijection), and similarly the pairs
\begin{align*}
(\dot \gamma_R(v_i), \dot \gamma_L(v_{i+1})) &\longleftrightarrow \beta_R(\bar v_i) \beta_L(\bar v_{i+1}), \\
(\gamma_R(u'_i), \gamma_L(u'_{i+1})) &\longleftrightarrow \beta_R(\bar u'_i) \beta_L(\bar u'_{i+1}) \\
(\dot \gamma_R(v'_i), \dot \gamma_L(v'_{i+1})) &\longleftrightarrow \beta_R(\bar v'_i) \beta_L(\bar v'_{i+1})
\end{align*}
are intercomputable. This argument is from \cite{Ka96}.

To clarify, the above permutations are applied to words of length $\ell$ on the bottom track, conditioned on $[ww]_{-12r}$ on the top track. The permutations are thus applied with an offset, and an individual application under an occurrence of $ww$ will not modify the $12r$ leftmost and rightmost symbols under the occurrence. In total at this step we modify all but the $12r$ left- and rightmost cells under a maximal occurrence of $w^m$.


Now, we deal with the remaining $12r$ coordinates under left corners of maximal occurrences $w^m$ by applying the (even) permutation
\begin{align*}
&\alpha_L(\gamma_L(u_1), \dot\gamma_L(v_1)) \alpha_L(\gamma_L(u'_1), \dot\gamma_L(v'_1)) \\
&\mapsto \beta_L(\bar u_1) \beta_L(\bar v_1) \beta_L(\bar u'_1) \beta_L(\bar v'_1) 
\end{align*}
of words of length $12r$ on the bottom track, conditioned on $[w]^c_{-\ell} \cap [w]_0 = [w]_0 \setminus [ww]_{-\ell}$ on the top track (the latter form shows that we have this controlled application in $G$). Here, observe that since the words $\bar u_i, \bar u'_i, \bar v_i, \bar v_i'$ were defined by applying $f$ to a periodic point in a conveyor belt fashion, the word $\beta_L(\bar u_1) \beta_L(\bar v_1)$ can be deduced from $(\gamma_L(u_1), \dot\gamma_L(v_1))$, and similarly for the primed versions. We deal with the right borders similarly.

Finally, to obtain the correct contents under $w^m$, we only need to perform the position swap
\begin{align*}
&\beta_L(\bar u_i) \beta_L(\bar v_i) \beta_L(\bar u'_i) \beta_L(\bar v'_i) \cdot \beta_R(\bar u_i) \beta_R(\bar v_i) \beta_R(\bar u'_i) \beta_R(\bar v'_i) \\
&\mapsto \beta_L(\bar u_i) \beta_R(\bar u_i) \beta_L(\bar v_i) \beta_R(\bar v_i) \cdot \beta_L(\bar u_i') \beta_R(\bar u_i') \beta_L(\bar v_i') \beta_R(\bar v_i') \\
&= \bar u_i \bar v_i \cdot \bar u_i' \bar v'_i
\end{align*}
under each occurrence of $w$. Note that this permutation is even, as we are simply applying the permutation $(1)(2 \; 5 \; 3)(4 \; 6\; 7)(8)$ to the ordering of words of length $3r$.
\end{proof}

\input{Lemma7FigureCM}

\thmMain*

\begin{proof}
Lemma~\ref{lem:FiveCase} implies that for any unbordered $w$, any $\pi \in \Alt(C^{|w|})$ controlled by any $|w|$-unbordered clopen set is in $\PAut[B; C]$, in particular this is true for the clopen sets $[w]_i$ and $[ww]_i$ for any $i$. From Lemma~\ref{lem:EvensUnderFIsEnough} we get that the groups $\langle \RCA_r(C) \rangle$ can be embedded for arbitrarily large $r$. Since $\RCA(C) = \bigcup_r \langle \RCA_r(C) \rangle$ and every finitely generated group of cellular automata over any alphabet embeds in $\RCA(C)$ \cite{KiRo90}, we conclude.
\end{proof}

Lemma~\ref{lem:EvensUnderFIsEnough} also directly applies to the commutator subgroup of $\RCA(B \times C)$ (since large enough alternating groups are perfect), so we also obtain that the commutator subgroup of $\RCA(B \times C)$, for any $|B|, |C| \geq 2$, is f.g.-universal. See Theorem~\ref{thm:Sofic} for a stronger result.

\subsection{The prime case}
\label{sec:Prime}

\begin{lemma}
\label{lem:Prime}
If $n \in \primes$, then $\PAut(n) \cong \langle \sigma \rangle \times S_n$.
\end{lemma}

\begin{proof}
Let $|A| = n$ and observe that $\PAut(A) = \PAut[A]$. The shift $\sigma$ commutes with symbol permutations, no symbol permutation is a non-trivial shift map on a full shift, and $\PAut[A]$ is by definition generated by symbol permutations and the shift $\langle \sigma \rangle$. Thus, the shift and the symbol permutations form a complementary pair of subgroups in $\PAut[A]$, and thus $\PAut[A]$ is an internal direct product of $\langle \sigma \rangle$ and the symbol permutations, which form a finite group isomorphic to $\Sym(A)$.
\end{proof}

\subsection{The linear case}
\label{sec:Linear}

By Lemma~\ref{lem:Prime}, $\PAut(A)$ is linear (even over $\R$) for somewhat uninteresting reasons when $|A|$ is prime. The case $|A| = 4$ gives a linear group as well, but a more interesting one. The crucial point is that all permutations of $\Z_2^2$ are affine, so all symbol permutations are ``affine''. 

Write $\Z_2[\pmb x, \pmb x^{-1}]$ for the ring of Laurent polynomials over the two-element field $\Z_2$. Write $\Z_2((\pmb x))$ for the field of formal Laurent series over $\Z_2$ (with only finitely many negative powers of $\pmb x$), which contains the ring $\Z_2[\pmb x, \pmb x^{-1}]$. For any (commutative unital) ring $R$, write $\GL(n, R)$ for the group of invertible $n$-by-$n$ matrices over $R$.
<<
\begin{theorem}
\label{thm:Four}
The group $\PAut(4)$ is linear, and has a faithful $8$-dimensional representation over $\Z_2((\pmb x))$. In fact,
\[ \PAut(4) \cong \Z_2^2 \rtimes \GL(2, \Z_2[\pmb x, \pmb x^{-1}]). \]
\end{theorem}

\begin{proof}
We begin with the second claim.
By renaming, we may assume the Cartesian product decomposition is $A = \{(0,0), (0,1), (1,0), (1,1)\}$, and we give $A$ the $\Z_2^2$-structure that arises from bitwise addition modulo $2$ with respect to this decomposition. Give also $A^\Z$ the structure of an abelian group, by cellwise addition.

Consider maps of the form $x \mapsto f(x) + a^\Z$, where $a \in A$ and $f$ is a reversible linear cellular automaton in the sense that $f(x + y) = f(x) + f(y)$ for all $x, y \in A^\Z$. A straightforward computation shows that such maps form a subgroup $G$ of $\Aut(A^\Z)$. The subgroup $K$ of maps $x \mapsto x + a^\Z$ for $a \in A$ is isomorphic to $\Z_2^2$, and a direct computation shows that it is normal in $G$. The subgroup $H$ of reversible linear cellular automata is also a subgroup, and we have $KH = G$, $K \cap H = 1$. It follows that $G = K \rtimes H$ is an internal semidirect product.

We can in a standard way see $H$ as the group $\GL(2, \Z_2[\pmb x, \pmb x^{-1}])$, by writing the local rule of a cellular automaton $f$ satisfying $f(x + y) = f(x) + f(y)$ as a matrix, so $G \cong \Z_2^2 \rtimes \GL(2, \Z_2[\pmb x, \pmb x^{-1}])$.

The generators of $\PAut(4)$ are contained in $G$ since all symbol permutations of $A$ are affine, and conversely it is straightforward to show that linear symbol permutations and partial shifts are a generating set for $H$, see e.g.\ \cite{Ka00}, and the maps $x \mapsto x + a^\Z$ are among generators of $\PAut(4)$ as well. It follows that $\PAut(4) = G$.

For the first claim, since $[G : H] = 4$ and $H$ is a $2$-dimensional matrix group, where the entries can be seen to be in the field $\Z_2((\pmb x))$, the induced representation of $G$ is $8$-dimensional over the same field.
\end{proof}

The action $\phi$ of $\GL(2, \Z_2[\pmb x, \pmb x^{-1}])$ on $\Z_2^2$ is the following: Let
\[ h : \GL(2, \Z_2[\pmb x, \pmb x^{-1}]) \to \GL(2, \Z_2) \]
be the group homomorphism obtained by applying the ring homomorphism extending $\pmb x^i \mapsto 1$ in each entry. The action $\phi$ is the pullback of the natural action of $\GL(2, \Z_2)$ on $\Z_2^2$ through $h$.

The group $\PAut(A)$ contains free groups when $|A| = 4$, as shown in the next section. It also contains a copy of the lamplighter group $\Z_2 \wr \Z$ (actually two natural embeddings of it, one acting on the top track and one on the bottom).

\subsection{Non-linearity and non-amenability}
\label{sec:LinAme}

We prove that apart from trivial cases (where the group is virtually cyclic and thus linear over any field admitting invertible matrices of infinite order), none of $\PAut[B; C]$ are amenable, and $\PAut[2;2]$ is the only linear case. This follows from natural embeddings of groups of the form $(\Z/m\Z)^\omega * (\Z/k\Z)^\omega$, where $m \leq |B|, k \leq |C|$.

In all cases $|B|,|C| \geq 2$ except $\PAut[2; 2]$, \emph{all} groups of the form $G^\omega * H^\omega$ are in $\PAut[B; C]$ for finite groups $G, H$, by f.g.-universality and by closure properties by Theorem~\ref{thm:ClosureProps}, but we give the simple direct argument and explain why the groups $(\Z/m\Z)^\omega * (\Z/k\Z)^\omega$ are indeed typically not even subdirect products of linear groups. The embedding is by RCA with one-sided bineighborhoods, thus we also obtain these subgroups in $\Aut((B \times C)^\N)$.

For $G$ a group, write $G^\omega$ for the direct union of $G^n$ as $n \rightarrow \infty$ (with the natural inclusions). For groups $G, H$ write $G * H$ for their free product. 

\begin{lemma}
\label{lem:FreeProduct}
Let $|B| = m, |C| = k$. 
Let $G, H$ be abelian groups with $|G| \leq m$, $H \leq k$. Then $G^\omega * H^\omega \leq \PAut[B; C]$.
\end{lemma}

\begin{proof}
Let $B = \{0,...,m-1\}, C = \{0,...,k-1\}$. The assumption $|G| \leq m$, $H \leq k$ is equivalent to the assumption that $G$ and $H$ act on $B$ and $C$, respectively, with at least one free orbit. Fix such actions. By renaming, we may assume $1 \in \{0,...,m-1\}$ and $1 \in \{0,...,k-1\}$ are representatives of the free orbits of $G$ and $H$, respectively.

The group $G^\omega$ is generated by the following maps: for $g \in G$ and $i \in \Z$, define
\[ f_{g, i}(x, y)_0 = \left\{\begin{array}{ll}
(g(x_0), y_0) & \mbox{if } y_{-i} = 1. \\
(x_0, y_0) & \mbox{otherwise}
\end{array}\right. \]
Extend $f_{g, i}$ to a cellular automaton by shift-commutation. These maps are easily seen to be in $\PAut[B; C]$, as $f_{g, 0}$ is a symbol permutation and the others are conjugate to it by partial shifts. Clearly we obtain a copy of $G$ by fixing $i$. Varying $i$, the maps commute since $G$ is abelian. By applying them to $(0^\Z, {^\omega}0.10{^\omega})$ we see that they do not satisfy any additional relations, and thus we have a copy of $G^\omega$. Define similarly $f_{h, i}$ for $h \in H$, by exchanging the roles of the tracks.

Of course restricting $i$ to $\N_+$, the maps $f_{g, i}$ and $f_{h, i}$ still give copies of $G^\omega$ and $H^\omega$, respectively. Denote these copies by $G' \cong G^{\omega}$ and $H' \cong H^{\omega}$. We show that together they satisfy no other relations, that is, the maps $f_{g,i}, f_{h,i}$ for $i > 0$ generate a copy of $G' * H' \cong G^\omega * H^\omega$.

Suppose that $f_w = f_\ell \circ \cdots \circ f_2 \circ f_1$ is a reduced element where $f_i \in G'$ for odd $i$, $f_i \in H'$ for even $i$, and that $\ell$ is even (the other three cases are completely symmetric). For each odd $i$ there is a ``maximal'' copy of $G$ used by $f_i$, i.e. the reduced presentation of $f_i$ contains some $f_{g_i, r_i}$ with $r_i \geq 1$ maximal and $g_i \in G \setminus \{1_G\}$. Similarly, for even $i$ there is some maximal copy of $H$ used, denote $r_i \geq 1$, $h_i \in H \setminus \{1_H\}$.

Now, a direct computation shows the $f_w$ acts non-trivially on the following configuration:
\[ {^{\omega}}\!\sm{0}{0} \sm{0}{1}\sm{0}{0}^{r_1-1} \sm{g_1^{-1} \cdot 1}{0}\sm{0}{0}^{r_2-1} \sm{0}{h_2^{-1} \cdot 1}\sm{0}{0}^{r_3-1} \cdots \sm{0}{0}^{r_\ell-1}\sm{0}{h_{\ell}^{-1} \cdot 1}\sm{0}{0}^{\omega} \]
For this, observe that $g_i^{-1} \cdot 1 \neq 1$ and $h_i^{-1} \cdot 1 \neq 1$ since $1$ is a representative of the free orbit on both tracks, and thus the rightmost ``active'' $1$ moves to the right on each step.
\end{proof}

\begin{lemma}
\label{lem:NonAmenable}
Let $G$ and $H$ be non-trivial groups. Then $G^\omega * H^\omega$ is not amenable.
\end{lemma}

\begin{proof}
A stronger fact is true: a free product of two non-trivial groups $G, H$ does not contain the free group on two generators if and only if it is amenable if and only if it is virtually cyclic if and only if $G \cong H \cong \Z_2$. Namely, $\Z_2 * \Z_2$ is the infinite dihedral group, which is virtually cyclic. If $g, g' \in G \setminus \{1_G\}$, $g \neq g'$ and $h \neq 1_H$, then $gh$ and $g'h$ freely generate a free group by the normal form theorem of free products \cite{LySc15}.
\end{proof}

The following lemma is classical. We give a direct proof mimicking \cite[Theorem~8.1.11]{Ro96} as suggested by user Panurge on the MathOverflow website~\cite{Pa16}.

\begin{lemma}
\label{lem:BoundedExp}
If $G$ is a linear $p$-group over a field of characteristic $q \neq p$, then $G$ is finite. The order of $G$ is at most $e^{d^3}$ where $e$ is the exponent and $d$ the dimension of the vector space.
\end{lemma}

\begin{proof}
We may assume $G$ acts on a vector space $V$ of dimension $d$ over an algebraically closed field $F$. Suppose $g^e = 1$ for all $g \in G$, where $e$ is a power of $p$. It follows from $g^e = 1$ that each root of the characteristic polynomial of $g$ is an $e$th root of unity (consider for example the Jordan normal form of $g$). There are at most $e$ such roots $\lambda_1, \ldots, \lambda_{e'}$, $e' \leq e$, so there are at most $e^d$ choices for the trace $\tr(g) = \sum_{j = 1}^d \lambda_{i_j}$ of any element of $g \in G$.

Suppose first that $G$ is irreducible. In this case by \cite[Theorem~8.1.9]{Ro96}, the fact that elements of $G$ have finitely many possible traces implies that $G$ itself is finite, in fact \cite[Theorem~8.1.9]{Ro96} gives the formula $|G| \leq (e^d)^{d^2} = e^{d^3}$.

Suppose next that $G$ is not irreducible. Then, there is a non-trivial subspace $U \leq V$ closed under the action of $G$. Suppose the dimension of $U$ is $d'$, so the dimension of $V/U$ is $d'' = d - d'$. By induction on dimension, and the fact the exponent cannot increase in subactions and quotients, the subgroup $L'$ ($L''$ resp.) of $G$ that acts trivially on $U$ ($V/U$ resp.) has index at most $e^{d'^3}$ (resp. $e^{d''^3}$). The subgroup $L = L' \cap L''$ of $G$ that fixes both $U$ and $V/U$ then has index at most $e^{d'^3} e^{d''^3} \leq e^{d^3}$. We will show that $L = 1$, which concludes the proof. 

Now, picking any basis of $m$ vectors for $U$ and extending it by $n$ vectors to a basis of $V$, we see that the corresponding matrix representation of $L$ (acting from the right) is by unitriangular matrices: each matrix is a block matrix of the form $\left( \begin{smallmatrix} I_n & N \\ 0 & I_m \end{smallmatrix} \right)$
where $I_m, I_n$ are the $m \times m$ and $n \times n$ identity matrices, respectively, and $N$ is an $n \times m$ matrix.

Suppose now $M \neq I_{m+n}$ is a unitriangular matrix over a field of characteristic $q$, and has order dividing $e$. 
Let $i$ be the leftmost column of $M$ containing a nonzero off-diagonal entry.

Now clearly the exponent of $M$, acting on the subspace of row vectors where all but the $i$ leftmost coordinates are $0$, is divisible by $q$. Thus the exponent of $M$ on the whole space is also divisible by $q$. Since the order of $M$ divides $e$, a power of $p$, the order of $M$ must be $1$, which is a contradiction with $M \neq 1$. This means we must indeed have $L = 1$.
\end{proof}

\begin{lemma}
\label{lem:FreeProductNonLinear}
Let $G$ and $H$ be non-trivial finite groups. If $G$ and $H$ are not $p$-groups for the same prime $p$, then $G^\omega * H^\omega$ is not a subdirect product of finitely many linear groups.
\end{lemma}

In particular the assumption includes the case where one of $G, H$ is not a $p$-group for any $p$.

\begin{proof}
The assumption implies that $p \vert |G|, q \vert |H|$ for some distinct primes $p, q$, so by Cauchy's theorem there exist $g \in G$, $h \in H$ such that $\ord(g) = p, \ord(h) = q$. It is then enough to prove that $\Z_p^\omega * \Z_q^\omega$ is not a subdirect product of finitely many linear groups.

Suppose it is, and let $\Z_p^\omega * \Z_q^\omega \cong K \leq G_1 \times G_2 \times \cdots \times G_\ell$ where the $G_i$ are linear groups. Let the characteristics of the underlying fields be $p_1, \ldots, p_\ell$, respectively. Let $I_p, I_q \subset \{1, \ldots ,\ell\}$ be defined by $i \in I_p \iff p_i = p$ and $i \in I_q \iff p_i = q$. Let $\pi_i$ be the natural projection $\pi_i : K \to G_i$.

By the previous lemma, $\pi_i(\Z_p^\omega)$ is finite for $i \notin I_p$. Thus, the intersection of the kernels of all these maps is some $K_p \leq \Z_p^\omega$ of finite index, in particular $K_p$ is non-trivial. Similarly we have a finite-index subgroup $K_q \leq \Z_q^\omega$. Then $K_p, K_q \leq K$ commute, which is a contradiction, since the subgroup they generate should be a free product $K_p * K_q \leq K$.
\end{proof}

The previous lemma implies in particular that a free product of linear groups need not be linear (or even a subdirect product of finitely many linear groups) when the characteristics of the fields over which they are linear are distinct, since the group $\Z_p^\omega$ is a linear group for every prime $p$ (for example a linear group of RCA by a matrix implementation of Lemma~\ref{lem:FreeProduct}). By \cite{Ni40} (see also \cite{We73}), the group $G^\omega * H^\omega$ is linear if and only if $G^\omega$ and $H^\omega$ are both linear over a field of the same characteristic.

\begin{theorem}
\label{thm:NonLinearNonAmenable}
If $|B|, |C| \geq 2$, then $\PAut[B; C]$ is non-amenable. If further $|C| \geq 3$, then $\PAut[B; C]$ is not a subdirect product of finitely many linear groups.
\end{theorem}

\begin{proof}
For non-linearity, if $|B| \geq m$ and $|C| \geq k$, then $(\Z/m\Z)^\omega * (\Z/k\Z)^\omega \leq \PAut[B; C]$ by Lemma~\ref{lem:FreeProduct}. If $m = k = 2$, Lemma~\ref{lem:NonAmenable} gives non-amenability. If $m = 2, k = 3$, Lemma~\ref{lem:FreeProductNonLinear} gives the second claim.
\end{proof}


\begin{proposition}
\label{prop:FreeProductAN}
Let $A = B \times C$, and let $G, H$ be abelian groups with $|G| \leq |B|$ and $|H| \leq |C|$. Then
$G^\omega * H^\omega \leq \Aut((B \times C)^\N)$.
\end{proposition}

\begin{proof}
In the construction of Lemma~\ref{lem:FreeProduct}, the generators are involutions and their neighborhoods are contained in $-\N$. Flipping the neighborhoods does not change the group, and gives reversible maps in $\Aut(A^\N)$.
\end{proof}

For $|A| \geq 8$, $\Aut(A^\N)$ is non-linear, as it does not even satisfy the Tits alternative \cite{Sa19a}. By the previous proposition, $\Aut(A^\N)$ is also non-linear for $|A| = 6$.

\subsection{Modifying just one track}
\label{sec:ModOneTrack}

The proof of Lemma~\ref{lem:ControlGrowing} implements the maps $\ctrl{\phi}{F}$ by elements of $\PAut[B; C]$ which only modify the bottom track. This is an interesting example of a finitely generated subgroup of $\PAut(A)$, for any alphabet $A \notin \primes \cup \{4\}$. Out of general interest, we take a brief look at its structure, which is much easier to understand than that of $\PAut(A)$.

This provides a new proof of the two-sided case of \cite{Sa19a}.

\begin{proposition}
\label{prop:PPAutSimulation}
Let $|B|, |C| \geq 2$ and let $\RPAut[B; C] \leq \PAut[B; C]$ be the subgroup generated by the partial shift on the bottom track, and symbol permutations that only modify the bottom track. Then $\RPAut[B; C] \cong P(B^\Z, \Sym(C))$.
\end{proposition}

\begin{proof}
Clearly the group $\RPAut[B;C]$ does not change if we replace the partial shift on the bottom track by the one on the top track.
Observe also that every cell on the bottom track behaves independently. The isomorphism simply tracks what happens at the origin.
\end{proof}

This motivates the study of the groups $P(B^\Z, H)$, especially when $H$ is a symmetric group.

\begin{proposition}
\label{prop:SolvableThree}
Let $|B| \geq 2$, let $H \leq \Sym(C)$ be a finite permutation group, and let $G = P(B^\Z, H)$. If $H$ has derived length $\ell$, then $G$ has derived length $\ell+1$. If $H$ is not solvable, $G$ is not virtually solvable.
\end{proposition}

\begin{proof}
Let $\phi : G \to \Z$ be the homomorphism that tracks the movement of the top track. Then $G$ is $\ker\phi$-by-$\Z$. Let $K = \ker\phi$, and observe that $[G, G] \leq K$ since $\Z$ is abelian.

Elements $g \in K$ do not modify the ``controlling configuration'' $B^\Z$ and only perform permutations on $C$ depending on the controlling word. Thus, $K$ is a subgroup of the uncountable direct product $H^{\beth_1}$ where $\beth_1 = 2^{\aleph_0}$. Whenever every element of $[H, H]$ can be expressed as a bounded product of commutators, we have $[H^X, H^X] = [H, H]^X$ for any set $X$. It follows that when $H$ is finite, the derived length of $H^{\beth_1}$ is the same as that of $H$, so the derived length of $G$ is at most one more than the derived length of $H$.

On the other hand, $[G, G] \leq K$ contains a subgroup mapping homomorphically onto $H$: consider the elements $[\sigma, \ctrl{g}{[1]_0}]$ where $g$ runs over $G$. If $x = {^\omega}0.10{^\omega}$, then $[\sigma, \ctrl{g}{[1]_0}]$ acts as $g$ on $C$, so the homomorphism that maps elements of $K$ to their action under the controlling configuration $x$ is indeed surjective onto $H$. It follows that the derived length of $G$ is at least one more than that of $H$.

If $H$ is not solvable, $G$ is not virtually solvable since it has $H^n$ as a subquotient for all $n$, which can be seen by conjugating elements $\ctrl{g}{[1]_0}$ by shifts and considering the action on elements of the form $(\sigma^i(x), a)$ with again $x = {^\omega}0.10{^\omega}$. 
\end{proof}

\begin{corollary}
\label{cor:SolvableThree}
Let $|B|, |C| \geq 2$. Then $G = P(B^\Z, \Sym(C))$ is (locally finite)-by-$\Z$. If $|C| \in \{2, 3, 4\}$, the group has derived length $|C|$. If $|C| \geq 5$, it is not virtually solvable.
\end{corollary}

\begin{proof}
In the previous proof, it was observed that $G$ is $\ker\phi$-by-$\Z$, and the kernel of $\phi$ is clearly locally finite when $H$ is finite since $H^{\beth_1}$ is locally finite. Thus $G$ is (locally finite)-by-$\Z$. For the claims about derived length, observe that $S_2$ is abelian, $S_3$ is metabelian and $S_4$ has derived length three, while $S_n$ for $n \geq 5$ is non-solvable.
\end{proof}

\begin{proposition}
\label{prop:NonLinear}
If $|B| \geq 2, |C| \geq 3$, then $\RPAut[B; C]$ is not linear.
\end{proposition}

\begin{proof}
The group is easily seen to contain copies of $\Z_2^n$ and $\Z_3^n$ for arbitrarily large $n$, since conjugating $\ctrl{g}{[1]_0}$ where $g$ is a generator of $\Z_k$, by the shift, we obtain a commuting set of maps which generate an internal direct product of copies of $\Z_k$, and the action is faithful, by considering the points $(\sigma^i(x), a)$ with $x = {^\omega}0.10{^\omega}$. This contradicts Lemma~\ref{lem:BoundedExp} by setting $n > 3^{d^3}$ where $d$ is the degree of the purported representation.
\end{proof}

The group is never nilpotent: let $g \in \Sym(C)$ be arbitrary and let $g_0 = \ctrl{g}{[1]_0}$ and $g_{i+1} = [\sigma, g_i]$. Then $g_i({^\omega} 010 {^\omega}, a) = ({^\omega} 010 {^\omega}, g a)$ for all $i$ (and of course if $|C| \geq 3$ already $\Sym(C)$ is not nilpotent).

We recover the two-sided case of \cite{Sa19a}.

\begin{proposition}
\label{prop:NoTits}
If $|B| \geq 2, |C| \geq 5$, then $R[B; C]$ does not satisfy Tits' alternative.
\end{proposition}

\begin{proof}
When $|B| \geq 2, |C| \geq 5$, $P(B^\Z, \Sym(C))$ is (locally finite)-by-cyclic, thus elementary amenable, thus does not contain a free group on two generators. It is not virtually solvable by Corollary~\ref{cor:SolvableThree}.
\end{proof}

Note that the group $\RPAut[B; C]$ in Proposition~\ref{prop:PPAutSimulation} is not equal to the group $\RAut(B^\Z \times C^\Z) \cap \PAut[B; C]$ in general: the f.g.-universality proofs in fact build copies of f.g.-universal cellular automata groups precisely inside $\RAut(B^\Z \times C^\Z) \cap \PAut[B; C]$.

\section{Corollaries}



\subsection{The optimal radius for an f.g.-universal group of CA}

One interesting class of naturally occurring RCA groups is obtained by varying $(|A|, N)$ and studying the group $\langle \RCA_N(A^\Z) \rangle$ they generate. 

In the context of the present paper, one could concretely ask, for example, which of these groups are linear and which contain all finitely generated groups of cellular automata. As an immediate corollary of the main theorem, we obtain the minimal contiguous bineighborhood size and biradius for f.g.-universality, for all but finitely many alphabets.

\begin{theorem}
\label{thm:OptimalRadiusProof}
Let $n \geq 2$ and let $G_n = \langle \RCA_1(n) \rangle$. The group $G_2$ is virtually cyclic, while $G_n \leq \RCA(n)$ is f.g.-universal whenever $n \geq 6$ is composite, or $n \geq 36$. If $N = \{a, a+1\}$ for some $a$ then $\langle \RCA_N(n) \rangle$ is not f.g.-universal for any $n$.
\end{theorem}

\begin{proof}
In the case $|A| = 2$, $N = \{-1,0,1\}$ we obtain the so-called elementary cellular automata. It is known that the group generated by reversible elementary cellular automata is $\Z \times \Z/2\Z$, generated by the shift and the bit flip.

Let now $U$ be the set of all numbers $n$ such that $\langle \RCA_1(n) \rangle$ is f.g.-universal in $\RCA(n)$. By Theorem~\ref{thm:Main}, $U$ contains all composite numbers except possibly $4$, since $\PAut(n) \leq \langle \RCA_1(n) \rangle$.

Let now $k, m \in \N$ be arbitrary. Then if $|A| = n = k^2 + m$, we can decompose the alphabet $A$ as $A = B^2 \sqcup C$ where $|B| = k$. A radius-$1$ cellular automaton can treat elements of $C$ as walls (which are never modified), and use the elements of $B^2$ as two $B$-tracks, wrapping into a conveyor belt next to elements of $C$. From this we obtain an embedding of the group $\langle \RCA_1(k) \rangle$ in $\langle \RCA_1(n) \rangle$. Since $\langle \RCA(k) \rangle$ has the same subgroups as $\langle \RCA(n) \rangle$, the f.g.-universality of $\langle \RCA_1(k) \rangle$ in $\langle \RCA(k) \rangle$ then implies f.g.-universality of $\langle \RCA_1(n) \rangle$ in $\langle \RCA(n) \rangle$. Thus, $U^2 + \N \subset U$, so $6 \in U$ implies $[36, \infty) \subset U$.

For the last claim, consider a contiguous neighborhood of size $2$. Such a neighborhood is either entirely in $\N$ or in $-\N$, so if $f$ and $f^{-1}$ both have such a neighborhood for all generators, they can be seen as elements of $\Aut(A^\N)$. No subgroup of $\Aut(A^\N)$ contains every finite group \cite{BoFrKi90}, so such a group cannot be f.g.-universal.
\end{proof}

In general, as $|A|$ grows the subgroups of $\langle \RCA_{\{0,1\}}(A^\Z) \rangle$ range over all finitely generated groups of one-sided cellular automata by standard blocking arguments, so these groups can be very interesting, even though they are never f.g.-universal in $\Aut(A^\Z)$.

The last claim is only true for contiguous neighborhoods of size two, and the theorem does not apply to e.g. $N = \{-1, 1\}$. Indeed, for the purpose of group embeddings one can consider the case $N = \{-1, 1\}$ to be the case of ``radius-$1/2$ RCA'', and by a standard blocking argument (see \cite{MoBo97}) and with a little bit of work one can indeed generate f.g.-universal groups this way (for some alphabets).

\subsection{The minimal number of generators}

\begin{theorem}
\label{thm:MinimalGenerating}
Let $G' \in \{\Z * \Z_2, \Z_2 * \Z_2 * \Z_2\}$ and $m \geq 2$. Then there is a homomorphism $\phi : G' \to \RCA(m)$ such that $\phi(G')$ is f.g.-universal. 
\end{theorem}

\begin{proof}
First consider $G' = \Z * \Z_2$. It is enough to show the statement for some $m$. We let $B$ with $|B| \geq 2$ be arbitrary and $C = \{0, 1\}$ and use the alphabet $A = B \times C$, $m = |A|$. Let $H$ be any f.g.-universal f.g.\ subgroup of $\RCA(m)$. By Lemma~\ref{lem:EvensUnderFIsEnough}, for any large enough $\ell$ and unbordered word $|w| = \ell$, if $G \leq \RCA(B \times C)$ contains
\[ \ctrl{\pi}{[w]_i} \mbox{ and } \ctrl{\pi}{[ww]_i} \]
for all $\pi \in \Alt(\{0,1\}^\ell)$ and all $i \in \Z$, then $G$ contains a copy of $H$.

Now, let $w \in B^\ell$ be unbordered where $\ell$ is as above, and sufficiently large. We construct a $2$-generated group $G$ containing the maps $\ctrl{\pi}{[w]_i}$ and $\ctrl{\pi}{[ww]_i}$, such that one of our generators is an involution.

Let $F : \{0,1\}^n \to \{0,1\}^n$ be a function such that $F^2 = \id$ and defining $f : \{0,1\}^\Z \to \{0,1\}^\Z$ by $f(x.vy) = x.F(v)y$, the maps $\sigma^i \circ f \circ \sigma^{-i}$ generate the group of all self-homeomorphisms $g$ of $\{0,1\}^\Z$ for which there exists $m$ such that
\[ \forall x \in \{0,1\}^\Z: \forall |i| \geq m: g(x)_i = x_i \]
holds. Such $F$ exists, in fact by Theorem~3 in \cite{Sa18c} one can pick $n = 3$ and
\[ F(a \cdot b \cdot c) = \left\{ \begin{array}{ll}
a \cdot b \cdot c, & \mbox{ if } a = 0 \wedge c = 1, \\
a \cdot (1 - b) \cdot c, & \mbox{ otherwise.} \end{array} \right. \]
(where $\cdot$ denotes concatenation).

Our generators are the partial shift on the first track, i.e.\ $\sigma_1(x, y) = (\sigma(x), y)$, and the map $f_0 = \ctrl{F}{[w]_0}$ (so we need $\ell \geq n$). Let
\[ G = \langle \sigma_1, \ctrl{F}{[w]_0} \rangle. \]
Note that $f_i = \ctrl{F}{[w]_{-i}} = f_0^{\sigma_1^{-i}} \in G$.

Let $F'$ be a universal family of reversible gates in the sense of Lemma~\ref{lem:UniversalGates}, i.e.\ $F'$ is a finite set of even permutations of sets of the form $\{0,1\}^k$ such that every even permutation of $\{0,1\}^m$ for any large enough $m$ can be decomposed into application of permutations in $F'$ in contiguous subsequences $\{i, i+1, ..., i+k-1\}$ of the indices $\{0,1,...,m-1\}$ (indeed one can fix $k = 3$). Note that $\{F\}$ need not be such a set: we may need to use more than $m$ coordinates to build permutations of $\{0,1\}^m$ using translates of $F$.

For any $i$, since $w$ is unbordered and of length $\ell$, the maps $f_i, f_{i+1}, ..., f_{i+\ell-n}$ compose in the natural way, just like translates of $F$ inside $\{0,1\}^\ell$. By universality of $F$, as long as $\ell$ is large enough, the maps $\ctrl{f'}{[w]_{-i}}$, $f' \in F'$, are generated. By the universality property of $F'$, we then have $\ctrl{\pi}{[w]_i} \in G$ for all $\pi \in\Alt(\{0,1\}^\ell)$ for all $i \in \Z$.

Now, we need to show that also $\ctrl{\pi}{[ww]_i} \in G$. For this, pick a large \emph{mutually unbordered} set $U \subset \{0,1\}^\ell$, i.e.\ any set such that $u_1, u_2 \in U$ have no nontrivial overlaps. For example we can pick $U = 0^{\ell-k-2} 1 \{0,1\}^k 1$
for any $k$ such that $k < \frac{\ell-4}{2}$. By the above, we can perform any even permutation of $U$ under occurrences of $w$. For two permutations $\pi_1, \pi_2 \in \Alt(\{0,1\}^\ell)$, with supports contained in $U$, a direct computation shows
\[ [ \ctrl{\pi_1}{[w]_i}, \ctrl{\pi_2}{[w]_{i + \ell}} ] = \ctrl{[\pi_1, \pi_2]}{[ww]_i}, \]
so for $|U| \geq 5$ (take $\ell$ large enough so that $2^{\lfloor \frac{\ell-4}{2}\rfloor - 1} \geq 5$) we have
$\ctrl{\pi}{[ww]_i} \in G$
for all $\pi \in \Alt(\{0,1\}^\ell)$ with support contained in $U$.

For two permutations $\pi_1, \pi_2 \in \Alt(\{0,1\}^\ell)$, a direct computation shows
\[ (\ctrl{\pi_1}{[ww]_i})^{\ctrl{\pi_2}{[w]_i}} = \ctrl{(\pi_1^{\pi_2})}{[ww]_i} \]
so, since $\Alt(\{0,1\}^\ell)$ is simple (supposing $\ell \geq 3$), $G$ in fact contains $\ctrl{\pi}{[ww]_i} \in G$ for all $\pi \in \{0,1\}^\ell$. This concludes the proof since $G$ is clearly a quotient of $G' = \Z * \Z_2$, as it was generated by an RCA of infinite order and an involution.

Let us then show the claim for $G' = \Z_2 * \Z_2 * \Z_2$. For this, pick $B = \{0,1\}$ and add a third component $B' = \{0,1\}$ on top, so the alphabet becomes $A = B' \times B \times C$, $m = 8$. Thinking of $x \in (B' \times B \times C)^\Z$ as having three binary tracks, and writing $\sigma_0$ and $\sigma_1$ for the shifts on the first two tracks, it is easy to see that $\sigma_0^{-1} \times \sigma_1$ is the composition of two involutions, say $\sigma_0^{-1} \times \sigma_1 = a \circ b$. Take $f_0$ the same map as before, but ignoring the $B'$-track entirely.

In the group $G = \langle a, b, f_0\rangle$ we now have the elements $\id_{B'^\Z} \times \pi^{[w]_i}$ and $\id_{B'^\Z} \times \pi^{[ww]_i}$ for any $w$ unbordered of large enough length $\ell$ and $i \in \Z$, by the same proof as above. Clearly the group they generate is isomorphic to the subgroup of $(B \times C)^\Z$ generated by $\pi^{[w]_i}$ and $\pi^{[ww]_i}$, thus Lemma~\ref{lem:EvensUnderFIsEnough} gives f.g.-universality.
\end{proof}

\subsection{Sofic shifts and the perfect core}

\begin{lemma}
Let $|B| = m, |C| = n$. Then the maps $(a,a',b,b') \mapsto (b,a',a,b')$ and $(a,a',b,b') \mapsto (a,b',b,a')$ are in $\Alt(B \times C \times B \times C)$ if and only if $2 | \binom{m}{2} n$ and $2|\binom{n}{2} m$.
\end{lemma}

\begin{proof}
The permutation $(a,a',b,b') \mapsto (b,a',a,b')$ is even if and only if the number of unordered pairs $\{(a,a',b,b'), (b,a',a,b')\}$ is even. The number of such pairs is $\binom{m}{2} n^2$. Symmetrically $(a,a',b,b') \mapsto (a,b',b,a')$ is even if and only if $2|\binom{n}{2} m^2$.
\end{proof}

\begin{lemma}
\label{lem:Commutator}
Suppose $m, n \geq 2$,
$2 | \binom{m}{2} n$ and $2|\binom{n}{2} m$.
Then $\PAut[m; n; m; n]$ has a perfect subgroup $G$ generated by six involutions, such that $\PAut[m; n] \hookrightarrow G$.
\end{lemma}



\begin{proof}
Let $|B| = m, |C| = n$ and $A = B \times C \times B \times C$. The symbol permutations $\updownarrow_B, \updownarrow_C$ defined by ${\updownarrow}_B(x,x',y,y') = (y,x',x,y')$ and ${\updownarrow}_C(x,x',y,y') = (x,y',y,x')$ are in $\Alt(A)$ under the conditions by the above lemma.
Define $\mathrlap{\swarrow}{\nearrow}_B = {\updownarrow}_B^{\sigma_1 \circ \sigma_2} = {\updownarrow}_B^{\sigma_1} \in \PAut[B; C; B; C]$ and $\mathrlap{\swarrow}{\nearrow}_ C= {\updownarrow}_C^{\sigma_1 \circ \sigma_2} = {\updownarrow}_C^{\sigma_2}  \in \PAut[B; C; B; C]$. 
Define also
\[ \sigma_B = [\updownarrow_B, \mathrlap{\swarrow}{\nearrow}_B] = \sigma_1^2 \circ \sigma_3^{-2} \in \PAut[B; C; B; C] \]
\[ \sigma_C = [\updownarrow_C, \mathrlap{\swarrow}{\nearrow}_C] = \sigma_2^2 \circ \sigma_4^{-2} \in \PAut[B; C; B; C] \]
For every symbol permutation $\pi \in \Sym(B \times C)$, the diagonal permutation $\pi \times \pi : A \to A$ is even. 


It is well-known that $\Alt(A)$ is generated by three involutions, so let $|F| = 3$ be any set of symbol permutations corresponding to such a generating set. Then $F \cup F^{\sigma_1 \circ \sigma_2}$ generates all of $\updownarrow_B$, $\updownarrow_C$, $\mathrlap{\swarrow}{\nearrow}_B$ and $\mathrlap{\swarrow}{\nearrow}_B$, thus it generates $\sigma_B$ and $\sigma_C$.

Now, it is easy to see that $\sigma_B$ and $\sigma_C$ and the symbol permutations $\pi \times \pi$ simulate four independent copies of $\PAut[B; C]$ in $\PAut[B; C; B; C]$: one in the even cells of the top track, one in the odd cells, and similarly two copies on the bottom track. Thus the group $G = \langle F \cup F^{\sigma_1 \circ \sigma_2} \rangle$ contains an embedded copy of $\PAut[B; C]$. Since $\Alt(A)$ is perfect, all the generators of $G$ can be written as a product of commutators of elements in $\Alt(A)$, so also $G$ is perfect.
\end{proof}

\begin{theorem}
\label{thm:Sofic}
Let $X$ be a sofic shift. Then the following are equivalent:
\begin{itemize}
\item The group $\Aut(X)$ has a perfect subgroup generated by six involutions containing every f.g.\ subgroup of $\Aut(A^\Z)$ for any alphabet $A$.
\item The group $\Aut(X)$ is not elementarily amenable.
\item $X$ has uncountable cardinality.
\end{itemize}
\end{theorem}

\begin{proof}
Suppose first that $X$ is uncountable. Standard embedding theorems \cite{KiRo90,Sa16a} show that $\Aut(A^\Z) \hookrightarrow \Aut(X)$ for any alphabet $A$. The choice $|B| = 2, |C| = 4$ satisfies the assumptions of Lemma~\ref{lem:Commutator}. Let $A = B \times C \times B \times C$, so that $\PAut[B; C]$ is f.g.-universal and contained in $\PAut(A)$. Let $G$ be the group provided by Lemma~\ref{lem:Commutator}. 
Then $G$ is a finitely generated perfect subgroup of $\PAut(A)$, generated by six involutions, which contains every group of cellular automata on any alphabet. We have $G \leq \PAut(A) \leq \Aut(A^\Z) \hookrightarrow \Aut(X)$. 

For any countable subshift $X$, $\Aut(X)$ is elementarily amenable by \cite{SaSc16a}, thus cannot contain a free group, thus cannot contain every finitely generated subgroup of $\Aut(A^\Z)$ for any non-trivial alphabet $A$. This paper is unpublished, but the case of countable sofics can be obtained by adapting \cite[Proposition 2]{Sa15}.
\end{proof}

Note that we do not claim that $\Aut(X)$ has an f.g.-universal f.g.\ subgroup for any $X$ other than a full shift. See Question~\ref{q:Sofic}.

The \emph{perfect core} $c(G)$ of a group $G$ is the largest subgroup $H$ such that $H = [H, H]$. The group $c(G)$ is contained in the commutator subgroup of $G$ and contains every perfect subgroup of $G$. Note that the conclusion of the previous theorem is stronger than simply finding an f.g.-universal f.g.\ subgroup of the perfect core, since a perfect group can contain non-perfect subgroups.

\subsection{The abstract statement}

\abstractstatement*

\begin{proof}
Pick $G \leq \PAut(64)$ as in the proof of Theorem~\ref{thm:Sofic}, so $G$ is finitely generated and perfect, and contains every finitely generated group of cellular automata on every alphabet.

Groups of RCA on full shifts are residually finite and f.g.\ groups of RCA have decidable word problems \cite{BoLiRu88}, so $G$ has these properties. The periodicity of RCA is undecidable \cite{KaOl08}. The f.g.-universality of $G$, together with the fact our proofs are algorithmic, then implies that it has an undecidable torsion problem.

Since the Tits alternative does not hold in $\Aut(A^\Z)$ \cite{Sa19a} and all f.g.\ graph groups are subgroups of $\Aut(A^\Z)$ \cite{KiRo90}, the same results hold for $\mathcal{G}$. The set $\mathcal{G}$ has the same closure properties as the set of subgroups of $\Aut(A^\Z)$, which by \cite{KiRo90} include finite extensions and by \cite{Sa16a} include direct products and free products.
\end{proof}

\subsection{Finitely subgenerated cellular automata groups}
\label{sec:NFGinFG}


We make some basic observations about which (not necessarily finitely generated) subgroups of $\Aut(A^\Z)$ can be embedded in our f.g.-universal f.g.\ groups, based on abstract arguments only.

For any group $G$, write $\mathrm{S}G$ for its set of subgroups, and write $\mathrm{SF}G$ for its \emph{finitely subgenerated subgroups}, i.e.\ those subgroups $H \leq G$ such that $H \leq K$ for some finitely generated subgroup of $G$. write $\mathcal{G}' = \mathrm{SF} \; \RCA(A)$ for some non-trivial alphabet $A$ (recall that this does not depend on $A$).

\begin{lemma}
Let $G$ be a group. We have $\mathrm{S} G = \mathrm{SF} G$ if and only if $G$ has a universal finitely generated subgroup, i.e.\ $G \hookrightarrow K \hookrightarrow G$ for some f.g.\ group $K$.
\end{lemma}

\begin{proof}
If $\mathrm{S} G = \mathrm{SF} G$, then since $G \in \mathrm{S} G$ there is a finitely generated subgroup $K \leq G$ containing $G$. If $G \hookrightarrow K \leq G$ then also $H \hookrightarrow K$ for all subgroups $H \leq G$.
\end{proof}

\begin{lemma}
Suppose $G$ has an f.g.-universal f.g.\ subgroup. If $\mathrm{S} G$ is closed under countable free products (resp. countable direct products), then so is $\mathrm{SF} G$. If $\mathrm{S} G$ is closed under direct products and finite extensions, then so is $\mathrm{SF} G$.
\end{lemma}

\begin{proof}
For finite direct and free products, the result follows since $K^n$ and $K * K * \cdots * K$ are finitely generated for any f.g.-universal f.g.\ group $K$. For infinite ones, observe that in particular $K * K \leq K$ and $K \times K \leq K$ for any f.g.-universal f.g.\ $K$, which implies that the set of subgroups of $K$ is also closed under countable free and direct products \cite{Sa16a}.

Every finite extension of a group $H$ is a subgroup of $H \wr S_n$ for large enough $n$, and conversely $H \wr S_n$ has $H^n$ as a finite-index subgroup. Suppose $H \in \mathrm{SF} G$, i.e. $H \leq K$ for an f.g.-universal f.g.\ $K$. Since $\mathrm{S} G$ is closed under direct products and finite extensions, the wreath product of $K$ by any symmetric group $S_n$ is in $\mathrm{SF} G$, thus $H \wr S_n \leq K \wr S_n \leq K$, implying that every virtually-$H$ group is in $\mathrm{SF} G$.
\end{proof}

\begin{theorem}
\label{thm:ClosureProps}
The class $\mathcal{G}'$ is closed under countable free and direct products and finite extensions.
\end{theorem}

From these closure properties, we obtain also that the free product of all finite groups, constructed as a CA group in \cite{Al88}, is in $\mathcal{G}'$.

We conjecture that all countable locally finite residually finite groups are in $\mathcal{G}'$, as it seems clear that the construction in \cite{KiRo90} can be performed directly. We do not know whether the group constructed in \cite{Br93} is in $\mathcal{G}'$.

\section{Questions}
\label{sec:Questions}

\subsection{Automorphism groups of full $\Z$-shifts}

The following question was mentioned in the introduction.

\begin{question}
\label{q:Universal}
Let $A$ be a non-trivial finite alphabet. Does $\Aut(A^\Z)$ have a finitely generated subgroup containing $\Aut(A^\Z)$ as a subgroup? Is the commutator subgroup $[\Aut(A^\Z), \Aut(A^\Z)]$ such a group?
\end{question}


The latter question is two questions in one: the author does not know whether the commutator subgroup is finitely generated (this has been previously asked in \cite{Sa17b}), and does not know whether it is a universal one. The question is also open at least for all transitive SFTs, but outside full shifts we do not even know when $\Aut(X)$ and $\Aut(Y)$ have the same subgroups (or even finitely generated subgroups).

In Theorem~\ref{thm:OptimalRadius}, we do not know the f.g.-universality status of $\langle \RCA_1(n) \rangle$ for
\[ n \in \{3,4,5,7,11,13,17,19,23,29,31\}. \]
We have not looked at these cases in detail.

Throughout the article, we have allowed the use of any symbol permutation. One obtains a large class of RCA groups by varying the permutation group allowed.

\begin{question}
\label{q:VaryingPermutationGroup}
Let $G \leq \Sym(B_1 \times B_2 \times ... \times B_k)$ be a permutation group. What can be said about the group $\PAut_G[B_1; B_2; \cdots; B_k]$ generated by partial shifts and symbol permutations in $G$?
\end{question}

If we restrict to even permutations, then for many alphabets, in particular whenever $|B|$ and $|C|$ are large enough, the arguments of the present paper can be used to establish f.g.-universality.


For a finitely generated group $G = \langle g_1, ..., g_k \rangle$, we say $f \in G$ is \emph{distorted} if $\langle f \rangle$ is infinite and satisfies $\mathrm{wn}(f^n) = o(n)$ where
\[ \mathrm{wn}(g) = \min \{\ell \;|\; \exists i_1,i_2,...,i_{\ell}: g = g_{i_1}g_{i_2}...g_{i_{\ell}} \} \]
It is open whether $\Aut(A^\Z)$ contains elements which are distorted in some finitely generated subgroup \cite{CyFrKrPe18} (a related question is asked in \cite{KiRo90}). Note that if $G$ is finitely generated and $f \in G$ is distorted in a subgroup $f \in H \leq G$, then $f$ is also distorted in $G$. Thus, by our main result, we can use $\PAut(A)$ as the canonical subgroup, and state the problem equivalently without quantification over f.g.\ subgroups:

\begin{question}
Does $\PAut(A)$ contain distortion elements for some $A$? 
\end{question}

By the universality result, the question stays equivalent if we fix $|A| = 6$. In \cite{CyFrKr19}, a notion of \emph{range-distortion} is defined. This notion is implied by distortion, and occurs in automorphism groups of all uncountable sofic shifts \cite{GuSa17}. Since our group-embeddings are by simulation, it is not hard to show that $\PAut(A)$ also contains range-distorted elements.

Finitely generated linear groups can contain distorted elements, as for example the discrete Heisenberg group (of invertible unitriangular $3 \times 3$ matrices over $\Z$) has distorted cyclic center. However, distortion cannot happen in linear groups over fields with positive characteristic by \cite[Lemma 2.10]{PuWu17}, so $\PAut(A)$ with $|A| = 4$ does not contain distortion elements. 

Two other questions we do not know the answer to are whether $\PAut(A)$ contains torsion (i.e. periodic) finitely generated infinite subgroups, or whether $\PAut(A)$ contains subgroups of intermediate growth, discussed previously in \cite{Sa19a}. Again $\PAut[2;2]$ cannot have such subgroups by linearity.

Another natural direction to take is to further study the poset $\mathcal{P}$ of finitely generated subgroups of $\Aut(A^\Z)$ up to embeddability (and identifying $G \approx H \iff G \hookrightarrow H \hookrightarrow G$). For example, this poset contains all finitely generated free groups as one element. This poset embeds in a natural way in the lattice $\mathcal{L}$ whose elements are subgroup- and isomorphism-closed collections of f.g.\ subgroups of $\Aut(A^\Z)$, under inclusion. The lattice $\mathcal{L}$ obviously has a maximal element, namely the family of all f.g.\ subgroups of $\Aut(A^\Z)$. Our main result states that this top element is actually in $\mathcal{P}$.


Finally, it would also be of interest to study universality for submonoids of $\End(A^\Z)$, the endomorphism monoid of $A^\Z$ consisting of all cellular automata under composition, taking the identity CA $\id$ as the monoid identity. The invertible part of a universal or f.g.-universal submonoid must then be universal or f.g.-universal in $\Aut(A^\Z)$, so this problem is at least as hard as the problem studied here.

One can also consider the semigroup of cellular automata without fixing an identity element, and define universality and f.g.-universality similarly, allowing any idempotent CA to play the role of the identity CA.

\subsection{Universality in other groups}

In this section, we ask universality questions for some of our favorite groups and make some basic observations. Of course, one can ask about universality in other groups, and we invite the reader to add their favorite groups to the list.

We begin by noting that there are some well-known non-finitely generated group that have universal finitely generated subgroups:

\begin{example}
The abelian group $(\Q^d, +)$ is not f.g.\ but $\Z^d$ is an f.g.-universal f.g.\ subgroup. On the other hand, $(\R^d, +)$ has no f.g.-universal f.g.\ subgroup or a countable universal subgroup (consider a Hamel basis). \qee
\end{example}

\begin{example}
\label{ex:Free}The (non-abelian) free group on $\aleph_0$ (free) generators has a universal finitely generated subgroup, namely the free group on two generators, since free groups with finitely or countably many generators all embed into each other. The free group on $\aleph_1$ generators does not have a universal finitely generated subgroup (since f.g.\ groups are countable), but the free group on two generators is an f.g.-universal subgroup of it, for the same reason as in the previous case. \qee
\end{example}

In the examples, the reason for non-universality was rather trivial (cardinality). Is there a countable group containing an f.g.-universal f.g.\ subgroup which is not universal, or (equivalently) is there one containing an f.g.-universal f.g.\ subgroup but no universal f.g.\ subgroup? We expect that the answers are positive, but do not know such examples (though $\Aut(A^\Z)$ could be an example for all we know).

The groups $\Aut(A^\N)$ for different $|A|$ have a different set of subgroups in general, as there are strong restrictions on even the finite subgroups \cite{BoFrKi90}. Thus, we cannot expect a finitely generated subgroup that contains a copy of every cellular automata group on every alphabet, unlike in the two-sided case. However, for a fixed alphabet we do not see a reason why f.g.-universality would not be possible. (The case $|A| = 2$ is trivial \cite{He69}.)

\begin{question}
Is there an (f.g.-)universal f.g.\ subgroup of $\Aut(A^\N)$ for some finite alphabet $|A| \geq 3$?
\end{question}

Very little is known about embeddings between automorphism groups of higher-dimensional subshifts, even two-dimensional full shifts, for example it is not known whether we can have $\Aut(A^{\Z^{d'}}) \leq \Aut(B^{\Z^d})$ for $d' > d$, $|A|,|B| \geq 2$, and whether $\Aut(\{0,1\}^{\Z^2})\leq \Aut(\{0,1,2\}^{\Z^2})$ (see \cite{HoOpen}). The following question seems to lead into similar problems.

\begin{question}
Let $d \geq 2$. Does $\Aut(A^{\Z^d})$ have an (f.g.-)universal f.g.\ subgroup?
\end{question}

Another obvious direction to look at are sofic shifts. For some simple sofics it is easy to show there are no f.g.-universal subgroups, and some even have finitely generated automorphism groups, but for most of them we have no idea. In particular, we do not know the answer for any mixing SFT which is not a full shift.

\begin{question}
\label{q:Sofic}
Let $X$ be a sofic shift. When does $\Aut(X)$ have an (f.g.)-universal f.g.\ subgroup?
\end{question}

This problem does not seem feasible at the moment: It is not known when $\Aut(Y)$ embeds into $\Aut(X)$ for mixing SFTs $X, Y$. Trying to find non-trivial self-embedding of subgroups of $\Aut(X)$ into $\Aut(X)$ runs into similar difficulties.

The author does not know another class of subshifts where such a universality question would be interesting. We note, however, that the non-f.g.\ automorphism groups of minimal subshifts constructed in \cite{BoLiRu88,Sa14d} both have an f.g.-universal f.g.-subgroups, namely $\langle \sigma \rangle$. In \cite{BoLiRu88}, $(\Q, +)$ is constructed, in \cite{Sa14d}, the dyadic rationals.

It is shown in \cite{BeHyMa19} that the asynchronous rational group (consisting of all asynchronous finite-state transductions defining a self-homeomorphism of $A^\N$, for a finite alphabet $A$) is not finitely generated, so one can ask for universality results. The set of subgroups of the asynchronous rational group does not depend on the alphabet.

As for synchronous automata groups, as with one-sided subshifts, one needs to fix a single alphabet, or even finite groups pose a problem for universality (since there is no boundedly-branching rooted tree where all finite groups act faithfully by automorphisms). When one alphabet is fixed, the group of all synchronous automata transductions is not finitely generated, as it has infinite abelianization (consider the signs of permutations performed on different levels or the tree).

\begin{question}
Is there an (f.g.-)universal automata group over a finite alphabet $A$? Does the asynchronous rational group have an (f.g.-)universal f.g.\ subgroup?
\end{question}

Especially in connection with Theorem 3.3 of \cite{BaKaNe10}, one could also ask whether there are universal automata groups within automata groups of bounded activity.

\begin{question}
Is there an (f.g.-)universal f.g.\ subgroup of the group of reversible Turing machines of \cite{BaKaSa16}?
\end{question}

A large finitely generated subgroup of ``elementary Turing machines'' is constructed in the planned extended version of \cite{BaKaSa16}, but the author does not know whether it is f.g.-universal.

Topological full groups are another class where such a question can be asked. It seems plausible that marker arguments can be used to prove universality results at least on full shifts. 

\begin{question}
Let $X$ be a subshift. When does the topological full group of $X$ have an (f.g.-)universal f.g.\ subgroup?
\end{question}

Some other groups with similar symbolic flavor are Thompson's $V$ \cite{CaFlPa96} and $2V$ \cite{Br04a}, but these groups are finitely generated.

All the groups considered above of course act on Cantor space. The homeomorphism group of Cantor space or any manifold of positive finite dimension is uncountable, and thus not finitely generated. The homeomorphism group of Cantor space contains uncountably many non-isomorphic f.g.\ subgroups, and thus cannot contain an f.g.-universal subgroup, but it is not immediately clear to the author what happens with, for example, manifolds of positive finite dimension.


\begin{question}
Let $X$ be a topological space. When does the homeomorphism group of $X$ contain an (f.g.)-universal f.g.\ subgroup?
\end{question}




\section*{Acknowledgements}

I have studied the linear part (in the CA sense) of $\PAut(A)$ for $|A| = 4$ with Pierre Guillon and Guillaume Theyssier, and the linear case of Lemma~\ref{lem:FreeProduct} is due to Theyssier. I thank Thibault Godin and Ilkka T\"orm\"a for several interesting discussions. Question~\ref{q:VaryingPermutationGroup} was suggested by Godin. I thank Ilkka T\"orm\"a for spotting some typos. I thank Laurent Bartholdi for pointing out that the automorphism group of a boundedly branching tree cannot contain copies of every finite group. The fact that the group of all synchronous automata transductions is not finitely generated was shown to the author by Ivan Mitrofanov, by studying orbits of eventually periodic points. I thank the anonymous referee for useful suggestions and corrections.

\bibliographystyle{plain}
\bibliography{../../../bib/bib}{}

\def\ocirc#1{\ifmmode\setbox0=\hbox{$#1$}\dimen0=\ht0 \advance\dimen0
  by1pt\rlap{\hbox to\wd0{\hss\raise\dimen0
  \hbox{\hskip.2em$\scriptscriptstyle\circ$}\hss}}#1\else {\accent"17
  #1}\fi}\def\cprime{$'$}
\begin{thebibliography}{10}

\bibitem{AaGrSc15}
Scott Aaronson, Daniel Grier, and Luke Schaeffer.
\newblock The classification of reversible bit operations.
\newblock {\em Electronic Colloquium on Computational Complexity}, (66), 2015.

\bibitem{Al88}
Roger~C. Alperin.
\newblock Free products as automorphisms of a shift of finite type.
\newblock 1988.

\bibitem{BaKaSa16}
Sebasti{\'a}n Barbieri, Jarkko Kari, and Ville Salo.
\newblock {\em The Group of Reversible Turing Machines}, pages 49--62.
\newblock Springer International Publishing, Cham, 2016.

\bibitem{BaKaNe10}
Laurent Bartholdi, Vadim~A. Kaimanovich, and Volodymyr~V. Nekrashevych.
\newblock On amenability of automata groups.
\newblock {\em Duke Math. J.}, 154(3):575--598, 09 2010.

\bibitem{BeHyMa19}
James Belk, James Hyde, and Francesco Matucci.
\newblock On the asynchronous rational group.
\newblock {\em Groups Geom. Dyn.}, 13(4):1271--1284, 2019.

\bibitem{Bo19}
Tim Boykett.
\newblock Closed systems of invertible maps.
\newblock {\em J. Mult.-Valued Logic Soft Comput.}, 32(5-6):565--605, 2019.

\bibitem{BoKaSa16}
Tim Boykett, Jarkko Kari, and Ville Salo.
\newblock {\em Strongly Universal Reversible Gate Sets}, pages 239--254.
\newblock Springer International Publishing, Cham, 2016.

\bibitem{BoFrKi90}
Mike Boyle, John Franks, and Bruce Kitchens.
\newblock Automorphisms of one-sided subshifts of finite type.
\newblock {\em Ergodic Theory Dynam. Systems}, 10(3):421--449, 1990.

\bibitem{BoLiRu88}
Mike Boyle, Douglas Lind, and Daniel Rudolph.
\newblock The automorphism group of a shift of finite type.
\newblock {\em Transactions of the American Mathematical Society}, 306(1):pp.
  71--114, 1988.

\bibitem{Br04a}
Matthew~G. Brin.
\newblock Higher dimensional {T}hompson groups.
\newblock {\em Geometriae Dedicata}, 108(1):163--192, 2004.

\bibitem{Br93}
Ezra Brown.
\newblock Periodic seeded arrays and automorphisms of the shift.
\newblock {\em Transactions of the American Mathematical Society},
  339(1):141--161, 1993.

\bibitem{CaFlPa96}
James~W. Cannon, William~J. Floyd, and Walter~R. Parry.
\newblock Introductory notes on {R}ichard {T}hompson's groups.
\newblock {\em Enseignement Math{\'e}matique}, 42:215--256, 1996.

\bibitem{CoQuYa16}
Ethan~M. Coven, Anthony Quas, and Reem Yassawi.
\newblock Computing automorphism groups of shifts using atypical equivalence
  classes.
\newblock {\em Discrete Anal.}, pages Paper No. 3, 28, 2016.

\bibitem{CyFrKr19}
Van Cyr, John Franks, and Bryna Kra.
\newblock The spacetime of a shift endomorphism.
\newblock {\em Transactions of the American Mathematical Society},
  371(1):461--488, 2019.

\bibitem{CyFrKrPe18}
Van Cyr, John Franks, Bryna Kra, and Samuel Petite.
\newblock Distortion and the automorphism group of a shift.
\newblock {\em Journal of Modern Dynamics}, 13(1):147, 2018.

\bibitem{CyKr16b}
Van Cyr and Bryna Kra.
\newblock The automorphism group of a minimal shift of stretched exponential
  growth.
\newblock {\em Journal of Modern Dynamics}, 10:483--495, 2016.

\bibitem{CyKr16a}
Van Cyr and Bryna Kra.
\newblock The automorphism group of a shift of subquadratic growth.
\newblock {\em Proceedings of the American Mathematical Society},
  144(2):613--621, 2016.

\bibitem{Di69}
John~D Dixon.
\newblock The probability of generating the symmetric group.
\newblock {\em Mathematische Zeitschrift}, 110(3):199--205, 1969.

\bibitem{DoDuMaPe16}
Sebastian Donoso, Fabien Durand, Alejandro Maass, and Samuel Petite.
\newblock On automorphism groups of low complexity subshifts.
\newblock {\em Ergodic Theory and Dynamical Systems}, 36(01):64--95, 2016.

\bibitem{DoDuMaPe17}
Sebastian Donoso, Fabien Durand, Alejandro Maass, and Samuel Petite.
\newblock On automorphism groups of toeplitz subshifts.
\newblock {\em Discrete Analysis}, 19, 2017.

\bibitem{Fr15}
Stefan Friedl.
\newblock An introduction to 3-manifolds and their fundamental groups.
\newblock {\em Preprint}, 2015.

\bibitem{FrScTa19}
Joshua Frisch, Tomer Schlank, and Omer Tamuz.
\newblock Normal amenable subgroups of the automorphism group of the full
  shift.
\newblock {\em Ergodic Theory Dynam. Systems}, 39(5):1290--1298, 2019.

\bibitem{GAP}
The GAP~Group.
\newblock {\em {GAP -- Groups, Algorithms, and Programming, Version 4.9.2}},
  2018.

\bibitem{Gr87}
Mikhael Gromov.
\newblock Hyperbolic groups.
\newblock {\em Essays in group theory}, 8(75-263):2, 1987.

\bibitem{GuSa17}
Pierre Guillon and Ville Salo.
\newblock {\em Distortion in One-Head Machines and Cellular Automata}, pages
  120--138.
\newblock Springer International Publishing, Cham, 2017.

\bibitem{He69}
Gustav~A. Hedlund.
\newblock Endomorphisms and automorphisms of the shift dynamical system.
\newblock {\em Math. Systems Theory}, 3:320--375, 1969.

\bibitem{HoOpen}
Michael Hochman.
\newblock Groups of automorphisms of {SFT}s.
\newblock
  URL:http://math.huji.ac.il/\textasciitilde{}mhochman/problems/automorphisms.pdf
  (version: 2018-08-07).

\bibitem{Pa16}
Panurge (https://math.stackexchange.com/users/72877/panurge).
\newblock Does there exist such a subgroup of a linear group?
\newblock MathOverflow.
\newblock
  URL:https://math.stackexchange.com/questions/1891722/does-there-exist-such-a-subgroup-of-a-linear-group
  (version: 2018-07-06).

\bibitem{Ka96}
Jarkko Kari.
\newblock Representation of reversible cellular automata with block
  permutations.
\newblock {\em Theory of Computing Systems}, 29:47--61, 1996.
\newblock 10.1007/BF01201813.

\bibitem{Ka00}
Jarkko Kari.
\newblock Linear cellular automata with multiple state variables.
\newblock In {\em S{TACS} 2000 ({L}ille)}, volume 1770 of {\em Lecture Notes in
  Comput. Sci.}, pages 110--121. Springer, Berlin, 2000.

\bibitem{KaOl08}
Jarkko Kari and Nicolas Ollinger.
\newblock Periodicity and immortality in reversible computing.
\newblock In {\em Proceedings of the 33rd international symposium on
  Mathematical Foundations of Computer Science}, MFCS '08, pages 419--430,
  Berlin, Heidelberg, 2008. Springer-Verlag.

\bibitem{KiRo90}
K.~H. Kim and F.~W. Roush.
\newblock On the automorphism groups of subshifts.
\newblock {\em Pure Mathematics and Applications}, 1(4):203--230, 1990.

\bibitem{KoZa07}
Dessislava~H Kochloukova and Pavel~A Zalesskii.
\newblock Tits alternative for 3-manifold groups.
\newblock {\em Archiv der Mathematik}, 88(4):364--367, 2007.

\bibitem{LiMa95}
Douglas Lind and Brian Marcus.
\newblock {\em An introduction to symbolic dynamics and coding}.
\newblock Cambridge University Press, Cambridge, 1995.

\bibitem{Lo02}
M.~Lothaire.
\newblock {\em Algebraic combinatorics on words}, volume~90 of {\em
  Encyclopedia of Mathematics and its Applications}.
\newblock Cambridge University Press, Cambridge, 2002.

\bibitem{LySc15}
R.C. Lyndon and P.E. Schupp.
\newblock {\em Combinatorial Group Theory}.
\newblock Classics in Mathematics. Springer Berlin Heidelberg, 2015.

\bibitem{MoBo97}
Cristopher Moore and Timothy Boykett.
\newblock Commuting cellular automata.
\newblock {\em Complex Systems}, 11(1):55--64, 1997.

\bibitem{Ni40}
V.L. Nisnewitsch.
\newblock {\"U}ber {G}ruppen, die durch {M}atrizen \"uber einem kommutativen
  {F}eld isomorph darstellbar sind.
\newblock {\em Matematicheskii Sbornik}, 50(3):395--403, 1940.

\bibitem{Ol13}
Jeanette Olli.
\newblock Endomorphisms of sturmian systems and the discrete chair substitution
  tiling system.
\newblock {\em Dynamical Systems}, 33(9):4173--4186, 2013.

\bibitem{PuWu17}
Timm~von Puttkamer and Xiaolei Wu.
\newblock {Linear Groups, Conjugacy Growth, and Classifying Spaces for Families
  of Subgroups}.
\newblock {\em International Mathematics Research Notices},
  2019(10):3130--3168, 09 2017.

\bibitem{Ro96}
D.~Robinson.
\newblock {\em A Course in the Theory of Groups}.
\newblock Graduate Texts in Mathematics. Springer New York, 1996.

\bibitem{Sa15}
Ville Salo.
\newblock Groups and monoids of cellular automata.
\newblock In Jarkko Kari, editor, {\em Cellular Automata and Discrete Complex
  Systems}, volume 9099 of {\em Lecture Notes in Computer Science}, pages
  17--45. Springer Berlin Heidelberg, 2015.

\bibitem{Sa16a}
Ville Salo.
\newblock A note on subgroups of automorphism groups of full shifts.
\newblock {\em Ergodic Theory and Dynamical Systems}, page 1–13, 2016.

\bibitem{Sa14d}
Ville Salo.
\newblock {Toeplitz subshift whose automorphism group is not finitely
  generated}.
\newblock {\em Colloquium Mathematicum}, 146:53--76, 2017.

\bibitem{Sa17b}
Ville Salo.
\newblock Transitive action on finite points of a full shift and a finitary
  {R}yan’s theorem.
\newblock {\em Ergodic Theory and Dynamical Systems}, pages 1--31, 2017.

\bibitem{Sa18c}
Ville {Salo}.
\newblock {Universal gates with wires in a row}.
\newblock {\em ArXiv e-prints}, September 2018.
\newblock Available at \url{https://arxiv.org/abs/1604.01646}. Accepted in
  Journal of Algebraic Combinatorics.

\bibitem{Sa19a}
Ville Salo.
\newblock No {T}its alternative for cellular automata.
\newblock {\em {Groups, Geometry and Dynamics}}, 13:1437–1455, 2019.

\bibitem{SaSc16a}
Ville Salo and Michael Schraudner.
\newblock Automorphism groups of subshifts through group extensions.
\newblock Preprint.

\bibitem{SaTo15d}
Ville Salo and Ilkka T\"orm\"a.
\newblock Block maps between primitive uniform and pisot substitutions.
\newblock {\em Ergodic Theory and Dynamical Systems}, 35:2292--2310, 10 2015.

\bibitem{Se16}
Peter {Selinger}.
\newblock {Reversible k-valued logic circuits are finitely generated for odd
  k}.
\newblock {\em ArXiv e-prints}, April 2016.
\newblock Available at \url{https://arxiv.org/abs/1604.01646}.

\bibitem{We73}
B.A.F. Wehrfritz.
\newblock Generalized free products of linear groups.
\newblock {\em Proceedings of the London Mathematical Society}, 3(3):402--424,
  1973.

\end{thebibliography}


\begin{thebibliography}{HD82}




\normalsize
\baselineskip=17pt


\bibitem[G07]{Gratzer} G. Gr\"atzer,
\emph{More Math into \LaTeX},
4th ed., Springer, Berlin, 2007.


\bibitem[HD82]{HillDow}  R. Hill and A. Dow,
\emph{A ground-breaking achievement},
J.~Differential Equations 15 (1982), 197--211.

\bibitem[K74]{Kow}  J. Kowalski,
\emph{A very interesting paper},  
in: Algebra, Analysis and Beyond (Nowhere, 1973),   
E.~Fox et al. (eds.),
Lecture Notes in Math. 867, Springer, Berlin, 1974, 115--124.

\bibitem[N80]{Nov} A. S. Novikov,
\emph{Another fascinating article},  
Uspekhi Mat. Nauk 23 (1980), no.~3, 112--134 (in Russian); 
English transl.: Russian Math. Surveys 23 (1980), 572--595.

\bibitem[R]{Russ} B. Russell,
\emph{A new theorem},
arXiv:0612.9876 (2006).

\end{thebibliography}

\end{document}